\documentclass[12pt,reqno]{amsart}
\usepackage{amssymb}
\usepackage{amsfonts}
\usepackage{amsmath}
\usepackage{mathrsfs}

\newtheorem{theorem}{Theorem}[section]
\newtheorem{lemma}[theorem]{Lemma}
\newtheorem{cor}[theorem]{Corollary}
\newtheorem{prop}[theorem]{Proposition}

\theoremstyle{definition}

\newtheorem{example}[theorem]{Example}
\newtheorem{defn}[theorem]{Definition}
\newtheorem{defns}[theorem]{Definitions}

\newtheorem{rmk}[theorem]{Remark}
\newtheorem{notn}[theorem]{Notation}
\newcommand{\qnum}[2]{[#1]_{#2}}
\DeclareMathOperator{\id}{id}
\DeclareMathOperator{\supp}{supp}
\DeclareMathOperator{\ch}{char} \DeclareMathOperator{\ann}{ann}
\DeclareMathOperator{\aut}{Aut}
\DeclareMathOperator{\chr}{char}
\newcommand{\Z}{{\mathbb Z}}
\newcommand{\C}{{\mathbb C}}
\newcommand{\N}{{\mathbb N}}
\newcommand{\Q}{{\mathbb Q}}
\newcommand{\F}{{\mathbb F}}
\newcommand{\R}{{\mathbb R}}
\newcommand{\Y}{\mathcal Y}
\newcommand{\X}{\mathcal X}
\newcommand{\q}{\mbox{$\overline{q}$}}

\begin{document}
\title{Simple ambiskew polynomial rings}

\author{David A. Jordan}
\address{School of Mathematics and Statistics\\
University of Sheffield\\
Hicks Building\\
Sheffield S3~7RH\\
UK}
\email{d.a.jordan@sheffield.ac.uk}
\author{Imogen E. Wells}



\begin{abstract}
We determine simplicity criteria in characteristics $0$ and $p$ for a ubiquitous class of iterated skew polynomial rings in two indeterminates over a base ring. One obstruction to simplicity is the possible existence of a canonical normal element $z$. In the case where this element exists we give simplicity criteria for the rings obtained by inverting
$z$ and the rings obtained by factoring out the ideal generated by $z$. The results are illustrated by numerous examples including higher quantized Weyl algebras and generalizations of some low-dimensional symplectic reflection algebras.
\end{abstract}

\maketitle

\section{Introduction}
The first Weyl algebra $A_1(\F)$ over a field $\F$ of characteristic $0$ is the best known example of a simple skew polynomial ring of the form $R[x;\delta]$ where $\delta$ is a derivation of $R$. It illustrates the well-known
result that, for an algebra $R$ over a field $\F$ of characteristic $0$, $R[x;\delta]$ is simple if and only if
$R$ contains no non-zero proper $\delta$-invariant ideal and the derivation $\delta$ is outer. See \cite[Theorem 3.2]{CozF}, \cite[Theorem 1.8.4]{McCR} or \cite[Proposition 2.1]{GW}. The algebra $A_1(\F)$ is obtained on taking $R=\F[y]$ and $\delta=\frac{d}{dy}$.

The situation for simple skew polynomial rings of the form $R[x;\alpha,\delta]$ where $\alpha$ is an endomorphism of $R$ and $\delta$ is an $\alpha$-derivation of $R$ is less well understood and
there are few documented examples in which $\alpha$ is not an inner automorphism.
One significant but difficult example, due to Cozzens \cite{Coz} and fully documented in \cite{CozF}, features a division ring in the role of $R$ and a non-surjective endomorphism in the role of $\alpha$ and has interesting asymmetric properties. Detailed examples with an automorphism $\alpha$ that is not inner are surprisingly difficult to track down in the literature. Results are also mostly restricted to division rings; see the survey in the introduction to the recent
paper \cite{oinertetal}. Here we find simplicity criteria for a class of iterated skew polynomial rings in two indeterminates $x$ and $y$ over a base ring $A$ which, for convenience, we assume to be an $\F$-algebra for a field $\F$. These have been studied, at a variety of levels of generalisation, in a sequence of papers including  \cite{itskew,warfmem,htone,fds,qweyl,ambi} and were given the name \emph{ambiskew polynomial ring} in \cite{ambi}.

For our purposes here, the construction of an ambiskew polynomial ring $R(A,\alpha,v,\rho)$ requires three commuting automorphisms $\alpha$, $\beta$, $\gamma$ of $A$, with $\beta=\alpha^{-1}\gamma$, a non-zero element $\rho$ of $\F$ and an element $v$ of $A$ such that $va=\gamma(a)v$ for all $a\in A$ and $v=\gamma(v)$.  The indeterminates satisfy the relations
$xy=\rho yx+v$ and, for all $a\in A$,
$ya=\alpha(a)y$ and $xa=\beta(a)x$. Thus $y$ is adjoined in the formation of $A[y;\alpha]$ and the adjunction of $x$
involves extending $\beta$ to an automorphism of $A[y;\alpha]$, with $\beta(y)=\rho y$, and then forming
$A[y;\alpha][x;\beta,\delta]$ where $\delta$ is a $\beta$-derivation of $A[y;\alpha]$ such that $\delta(A)=0$ and $\delta(y)=v$. If $v$ is central then $\beta=\alpha^{-1}$. Taking $A=\F$, $\alpha=\beta=\id_\F$ and $\rho=v=1$ gives the first Weyl algebra $A_1(\F)$ and iteration yields the higher Weyl algebras $A_n(\F)$, $n\geq 2$, where $A=A_{n-1}(\F)$, $\alpha=\beta=\id_A$ and $\rho=v=1$.

The development of the theory of ambiskew polynomial rings began in \cite{itskew} where
$R$ is commutative, $\rho=1$ and $v=u-\alpha(u)$ for some $u\in A$. Examples included
the enveloping algebra $U(sl_2)$, its quantization $U_q(sl_2)$ and
the coordinate ring $\mathcal{O}_q(M_2(\F))$ of quantum $2\times 2$ matrices. Each of these examples has a distinguished central element, namely the quantum determinant in $\mathcal{O}_q(M_2(\F))$ and the Casimir elements in $U(sl_2)$ and $U_q(sl_2)$. These are examples of a general phenomenon. In the construction presented in \cite{itskew}, the element $xy-u=yx-\alpha(u)$ is central and is called the \emph{Casimir element}.
The algebras covered in \cite{itskew} include the algebras considered by Smith in \cite{psmith} and that paper heavily influenced the subsequent development of the theory of ambiskew polynomial rings.

Arbitrary non-zero values for $\rho$ were introduced in \cite{htone,fds}, where $v$ became $u-\rho \alpha(u)$ and the Casimir element became $xy-u=\rho(yx-\alpha(u))$. Although not central when $\rho\neq 1$, $xy-u$ is normal.
Among the new examples included by this generalization were the quantized Weyl algebra $A_1^q(\F)$, where, for $q\in \F\backslash\{0,1\}$, $xy-qyx=1$, the dispin enveloping algebra, where $\rho=-1$, and an algebra introduced by Woronowicz \cite{wor} in the context of quantum groups. At this stage, the Weyl algebra $A_1(\F)$ was excluded and there were no simple examples, due to the existence of the normal Casimir element which can never be a unit. In \cite{jw}, $A$ remained commutative but no conditions were imposed on the element $v$. The condition that $v=u-\rho\alpha(u)$ for some $u\in A$, giving rise to a normal Casimir element, was named {\it conformal} and its negation was named {\it singular}. Thus singularity is a necessary condition for simplicity and $A_1(\F)$ is a singular ambiskew polynomial ring.

To allow for iteration,
and application to the higher quantized Weyl algebras arising from the quantum calculus of Maltsiniotis \cite{malt},
noncommutative coefficient rings where introduced in \cite{qweyl} but only in the conformal case with  $v=u-\rho\alpha(u)$ for some $u\in A$ such that $ua=\gamma(a)u$ for all $a\in A$.
The higher quantized Weyl algebra $A_n^{\Lambda,\q}$ will be specified in detail in Section 6 but a brief discussion of the second quantized Weyl algebra, which requires three non-zero parameters $q_1,q_2$ and $\lambda$, will set the pattern and illustrate the role of Casimir elements in iteration. Renaming the generators, the first quantized Weyl algebra $A_1^{q_1}(\F)$ is generated by $x_1$ and $y_1$ subject to the relation \begin{equation}\label{xyq1}
x_{1}y_{1}-q_1y_{1}x_{1}=1.\end{equation}
Let $v=1+(q_1-1)y_1x_1$ which, if $q_1\neq 1$, is a multiple of the first Casimir element and, for all $q_1$, satisfies
the equations $vy_1=q_1y_1v$ and $vx_1=q_1^{-1}x_1v$. The second quantized Weyl algebra $A_2^{\lambda,q_1,q_2}(\F)$
is constructed using the $\F$-automorphisms $\alpha$ and $\beta$ such that $\alpha(x_1)=\lambda x_1$, $\alpha(y_1)=
\lambda^{-1}y_1$, $\beta(x_1)=(q_1\lambda)^{-1}x_1$ and $\beta(y_1)=q_1\lambda y_1$. Thus $A_2^{\lambda,q_1,q_2}(\F)$
is generated by $x_1,y_1,x_2$ and $y_2$ subject to \eqref{xyq1} and the further relations
\begin{eqnarray*}
y_{2}x_{1}&=&\lambda x_{1}y_{2},\\
y_{2}y_{1}&=&\lambda^{-1}y_{1}y_{2},\\
x_{2}x_{1}&=&(q_1\lambda)^{-1} x_{1}x_{2},\\
x_{2}y_{1}&=&q_1\lambda y_{1}x_{2},\\
x_{2}y_{2}-q_2y_{2}x_{2}&=&1+(q_{1}-1)y_{1}x_{1}.\\
\end{eqnarray*}
When $q_1=q_2=\lambda=1$ this is the Weyl algebra $A_2(\F)$. The
element $v+(q_{2}-1)y_{2}x_{2}$, which is either a multiple of a second Casimir element or equal to $v$, is normal
and can be used for further iteration.

The higher generalized Weyl algebra $A_n^{\Lambda,\q}$, where $\q=(q_{1},\ldots,q_{n})$ is an $n$-tuple of elements
of $\F$, and $\Lambda=(\lambda_{i,j})$ is an appropriate $n\times n$ matrix over $\F$,
 has $2n$ generators
$y_1,x_1,\ldots,y_n,x_n$ and is constructed by iteration. It has $n$ distinguished normal elements
$v_1,v_2,\ldots,v_n$ which are distinct if each $q_i\neq 1$.
It is shown in \cite{qweyl} that if $q_i\neq 1$ for all $i$, the algebra
obtained by inverting these $n$ normal elements is simple. We shall see that if each $q_j=1$ and $\ch \F=0$ then the higher quantized Weyl algebra is itself simple. This follows from one of our main results, Theorem~\ref{r-thm0}, which says that if $\ch \F=0$ then the ambiskew polynomial ring $R(A,\alpha,v,\rho)$ is simple if and only if
\begin{enumerate}
\item $A$ has no non-zero proper ideal invariant under $\alpha$,
\item $R$ is singular, and
    \item for all $m\geq 1$, the element $\sum_{l=0}^{m-1}\rho^{l}\alpha^{l}(v)$ is a unit.
\end{enumerate}
A more complex criterion for non-zero characteristic appears in Theorem~\ref{r-thmp}.

When $\ch \F=0$, the simple quantized Weyl algebra $A_2^{\lambda,1,1}(\F)$ has a claim to being the most accessible example of a simple skew polynomial ring of the form
$B[x;\alpha,\delta]$ with $\alpha$ not inner.
Two competitors, whose simplicity follows from Theorem~\ref{r-thm0}, are skew versions of the first Weyl algebra that have not, to our knowledge, previously appeared in the literature. For these, view $\C$ as an $\R$-algebra with conjugation as an $\R$-automorphism $\alpha$. We shall see that the $\R$-algebra extension of $\C$ generated by $x$ and $y$ subject to the relations
\[xi=-ix,\quad yi=-iy,\quad xy-yx=1\]
is a simple ambiskew polynomial ring, as is the $\R$-algebra in which the third of these relations is replaced by $xy+yx=i$.

New examples of ambiskew polynomial rings have continued to emerge, including the down-up algebras of Benkart and Roby \cite{benkartroby}, the generalized down-up algebras of Cassidy and Shelton \cite{CassShel} and, most recently, the augmented down-up algebras of Terwilliger and Worawannatoi \cite{terwor}. To accommodate the non-Noetherian down-up algebras, the definition of ambiskew polynomial ring was amended to allow $\alpha$ to be non-bijective and $\rho$ to be zero but here we shall assume that $\alpha$ is bijective and $\rho=0$.

In a higher quantized Weyl algebra with each $q_i\neq 1$, the only barrier to simplicity is the existence of the normal Casimir elements. The localization obtained by inverting these elements is simple. In Section 4, we find a simplicity criterion for the localization of a conformal ambiskew ring at the powers of the Casimir element $z$ to be simple. The barrier can also be removed by factoring out the ideal generated by $z$ which yields a
generalized Weyl algebra in the sense of Bavula \cite{vlad1}. A mild generalization is needed to cover the possibility that $v$ is not central. In Section 5, we give a simplicity criterion generalizing simplicity criteria from \cite[6.1 Theorem]{primitive} and its subsequent generalization \cite[Theorem 4.2]{vlad5}. Section 6 presents the application of results from Sections 3, 4 and 5 to examples.

All the examples mentioned so far are domains but simple ambiskew polynomial rings may have zero-divisors. There has been much interest in the symplectic reflection algebras of Etingof and Ginzburg \cite{EG}. In common with enveloping algebras, quantized enveloping algebras and quantum matrices, it turns out that low-dimensional examples of these algebras are ambiskew polynomial rings or iterated ambiskew polynomial rings over the group algebra of a finite cyclic group. Two such examples are documented as Examples \ref{sra1} and \ref{sra2}. Although simplicity criteria for these algebras are known, they provide nice illustrations of Theorem~\ref{r-thm0} and it is interesting to see how they fit into this picture. In both cases the parameter $\rho$ is $1$ but we shall also discuss some similar examples in which $\rho\neq 1$.

Some of the results presented here appeared in the PhD
thesis \cite{imthesis} of the second author.

\section{Preliminaries}\label{prelims}
\begin{defns}
Let $\gamma$ be an automorphism of a ring $A$ and let $v\in A$. If
$vA=Av$ then $v$ is {\it normal} in $A$ and if $va=\gamma(a)v$ for
all $a\in A$ and $\gamma(v)=v$, we shall say that $v$ is
$\gamma$-{\it normal}. If $v$ is a regular normal element then $v$
is $\gamma$-normal for a unique automorphism $\gamma$, while $0$
is $\gamma$-normal for every automorphism $\gamma$ of $A$.

Let $\Gamma$ be a set of automorphisms of $A$. An ideal $I$ of $A$
is a $\Gamma$-{\it ideal} of $A$ if $\gamma(I)\subseteq I$ for all
$\gamma\in\Gamma$. The ring $A$ is $\Gamma$-{\it simple} if $0$
and $A$ are the only $\Gamma$-ideals of $A$. If
$\Gamma=\{\alpha\}$ is a singleton, we write $\alpha$-ideal and
$\alpha$-simple rather than $\{\alpha\}$-ideal and
$\{\alpha$\}-simple.
\end{defns}

\begin{defn}\label{defambi}
The level of generality of the construction of ambiskew polynomial
ring varies through the papers \cite{itskew,warfmem,htone,fds,qweyl,ambi}. Here we shall
work with a common generalization of the constructions given in
\cite{qweyl} and \cite{jw}.  Let $\F$ be a field, and let $A$ be an
 $\F$-algebra. Let $\rho\in
\F\backslash\{0\}$ and let $v$ be a $\gamma$-normal element of $A$
for some  $\F$-automorphism $\gamma$ of $A$. Let
$\alpha\in\aut_{\F}A$ be such that $\alpha\gamma=\gamma\alpha$ and let $\beta:=\alpha^{-1}\gamma=\gamma\alpha^{-1}$, so that
$\alpha\beta=\gamma=\beta\alpha$. Extend $\beta$ to an $\F$-automorphism of
$A[y;\alpha]$ by setting $\beta(y)=\rho y$. By \cite[Exercise 2ZC]{GW}, there is a
$\beta$-derivation $\delta$ of $A[y;\alpha]$ such that $\delta(A)=0$ and
$\delta(y)=v$. The {\it ambiskew polynomial ring} $R(A,\alpha,v,\rho)$ is the iterated skew polynomial ring
$A[y;\alpha][x;\beta,\delta]$. Thus \begin{eqnarray*}
ya&=&\alpha(a)y\;\text{ for all }a\in A,\\
xa&=&\beta(a)x\;\text{ for all }a\in A\text{ and }\\
xy&=&\rho yx+v.
\end{eqnarray*}
Strictly speaking, we should write $R(A,\alpha,\gamma,v,\rho)$. However if $v$ is regular then $v$ determines $\gamma$ and the only zero-divisor appearing as $v$ in any example that we shall consider is $0$, in which case we take $\gamma=\id_A$.
There is some symmetry in the roles of $x$ and $y$ inasmuch that $R(A,\alpha,v,\rho)$ can also be presented as
$A[x;\beta][y;\alpha,\delta^\prime]$, where $\alpha(x)=\rho^{-1}x$, $\delta^\prime(A)=0$ and $\delta^\prime(x)=-\rho^{-1}(v)$.
If $v$ is central then the actions of $x$ and $y$ on $A$ involve
twisting using $\alpha$ from opposite sides: $ya=\alpha(a)y$ and
$ax=x\alpha(a)$. This is the reason for the name.
\end{defn}
Let
$v^{(0)}:=0$ and $v^{(m)} :=\sum_{l=0}^{m-1}\rho^{l}\alpha^{l}(v)$ for
$m\in\N$. In particular $v^{(1)}=v$.  Each  $v^{(m)}$ is $\gamma$-normal and
it is easily checked, by induction, that, for $m\geq 0$,
\begin{eqnarray}
\label{skewcomm} xy^{m}-\rho^{m}y^{m}x & = & v^{(m)} y^{m-1}\;\text{ and}
\\
 x^{m}y-\rho^{m}yx^{m} & = & x^{m-1}v^{(m)} = \beta^{m-1}(v^{(m)})x^{m-1}.
\end{eqnarray}

\begin{notn}\label{w}
Let $w:=xy=\rho yx+v$.  For all $a\in A$,
\[wa=(xy)a=\beta\alpha(a)xy=\gamma(a)w\] and the subring of $R$
generated by $A$ and $w$ may be identified with $A[w;\gamma]$.
Note that, as $\gamma(v)=v$, $w$ commutes with $v$, so that $v$ is normal in
$A[w;\gamma]$. Also $\gamma$ extends to an $\F$-automorphism of $A[w;\gamma]$, with
$\gamma(w)=w$, and $v$ is $\gamma$-normal in $A[w;\gamma]$ as well as in $A$. The $\F$-automorphisms  $\alpha$ and $\beta$ can be extended to $A[w;\gamma]$ by
setting $\alpha(w)=\rho^{-1}(w-v)$ and
$\beta(w)=\alpha^{-1}(w)=\rho w+\alpha^{-1}(v)$.
Then $yw=\alpha(w)y$, $xw=\beta(w)x$ and the following identities, in which $m\geq 1$, are easily established by induction on $m$:
\begin{eqnarray}
 \alpha^{m}(w)& =&\rho^{-m}(w-v^{(m)});\label{amw}\\
\alpha^{-m}(w)&=&\rho^{m}w+\beta^{m}(v^{(m)});\\
x^{m}y^{m} &=&\prod_{l=0}^{m-1}\alpha^{-l}(w);\label{xmym} \\
y^{m}x^{m} &=&
\prod_{l=1}^{m}\alpha^{l}(w)\label{ymxm}.
\end{eqnarray}
The factors in the products on the right hand sides of \eqref{xmym} and \eqref{ymxm} commute so the products are well-defined.
\end{notn}

\begin{defns}
\label{Casimir}
Suppose that there is a $\gamma$-normal element $u\in A$ such that
$v=u-\rho\alpha(u)$. Then the element
$z:=xy-u=\rho(yx-\alpha(u))$ is such that
$zy=\rho yz$, $zx=\rho^{-1} xz$, $za=\gamma(a)z$ for all $a\in A$ and $zu=uz$. Hence
$z$ is $\gamma$-normal in $R$, where $\gamma$ is extended to an $\F$-automorphism of $R$ by setting $\gamma(y)=\rho y$ and
$\gamma(x)=\rho^{-1}x$.
If such an element $u$ exists then it will be called a {\it splitting} element
and we shall say that the $4$-tuple
$(A,\alpha,v,\rho)$ is {\it conformal} and that $R$ is a {\it conformal ambiskew polynomial ring}; otherwise $(A,\alpha,v,\rho)$ and $R$ are {\it singular}.
In the conformal case there need not be a unique splitting element.  In general, if $u$ is a splitting element and $u^\prime\in A$ then $u^\prime$ is a splitting element if and only if $u-u^\prime$ is $\gamma$-normal and
$\alpha(u-u^\prime)=\rho(u-u^\prime)$. For example if $\rho=1$, $\gamma=\id_A$ and
$u$ is a splitting element then $u+\lambda$ is a splitting element for all $\lambda\in \F$. If $u$ is not unique we shall adopt a convention of choosing one splitting element $u$ and fixing it subsequently.
We then refer to the element
$z:=xy-u=\rho(yx-\alpha(u))$ as the {\it Casimir element}
of $R$. In the conformal case, note that $A[z;\gamma]=A[w;\gamma]$, that
$\alpha(z)=\rho^{-1}z$ and that $v^{(m)}=u-\rho^m\alpha^m(u)$ for all
$m\geq 0$.
\end{defns}

\begin{lemma} Suppose that $v$ is regular in $A$ and let $u\in A$ be such that, for some $n\geq 1$,  $ua=\gamma^n(a)u$ for all
$a\in A$. Then $\gamma(u)=u$. In particular, if $ua=\gamma(a)u$ for all $a\in A$ and
$v=u-\rho\alpha(u)$ then $u$ is a splitting element. \label{guu}
\end{lemma}
\begin{proof}
In this situation,
$uv=\gamma^n(v)u=vu=\gamma(u)v$ so, by the regularity of $v$, $\gamma(u)=u$. In the particular case, $\gamma(u)=u$ is the remaining condition required for $u$ to be a splitting element
\end{proof}

\begin{rmk} The use of the word {\it conformal} here is a possible source of confusion, due to the conformal $SL_2$
algebras of \cite{lleb2}, which take their name from physical
considerations and happen to be ambiskew polynomial rings. In \cite{MDON}, the term {\it $J$-conformal} is
used for this reason.
\end{rmk}

\begin{notn}
For a ring $A$, the group of units of $A$ will be denoted $U(A)$.
\end{notn}

\section{Simplicity of $R(A,\alpha,v,\rho)$} \label{simpleR}
Throughout this section, $R$ will denote an ambiskew polynomial ring
$R(A,\alpha,v,\rho)$ where $A$ is an algebra over a field $\F$. The set
${\Y}=\{y^i\}_{i\geq1}$ is a $\beta$-invariant right and left Ore set in $A[y;\alpha]$ and,
by \cite[Lemma 1.4]{g}, it is a right and left Ore set in $R$. The localization
$R_{\Y}$ is generated, as a ring extension of $A$, by $y$, $y^{-1}$ and $w$
or, in the conformal case, by $y$, $y^{-1}$ and $z$. It can be identified with
$A[w;\gamma][y^{\pm 1};\alpha]$, where $\alpha(w)=\rho^{-1}(w-v)$, or, in the conformal case,
with $A[z;\gamma][y^{\pm 1};\alpha]$, where $\alpha(z)=\rho^{-1}z$.
Similarly
${\X}=\{x^i\}_{i\geq1}$ is a right and left Ore set in $R$ and the
localization $R_{\X}$ can be identified with
$A[w;\gamma][x^{\pm 1};\beta]$, with $\beta(w)=\rho w+\alpha^{-1}(v)$ or, in the
conformal case, with $A[z;\gamma][x^{\pm 1};\beta]$, with  $\beta(z)=\rho z$.

The following lemma, which in the Noetherian case is an immediate consequence of \cite[2.1.16(vi)]{McCR}, will be useful in handling the passage between $R$ and $R_{\X}$ and in similar situations later in the paper.
\begin{lemma}\label{yIy}
Let $B$ be a ring with a regular element $y$ such that ${\Y}:=\{y^i\}_{i\geq1}$ is a right and left Ore set
and let $C=B_{\Y}=B{\Y}^{-1}={\Y}^{-1}B$ be the localization of $B$ at ${\Y}$. If $C$ is simple and $I$ is a non-zero ideal of $B$ then $y^s\in I$ for some $s\geq 0$.
\end{lemma}
\begin{proof}
Let $J$ be the set of elements of $C$ that are finite sums of elements of the form $y^{-a}iy^{-b}$ where $i\in I$ and $a,b\geq 0$. As $y^{-a}iy^{-b}=y^{-(a+1)}(yi)y^{-b}=y^{-a}(iy)y^{-(b+1)}$, any element of $J$ can be written as a single term
$y^{-m}iy^{-n}$. Using the right (resp. left) Ore condition it is readily checked that $J$ is a right (resp. left) ideal of $C$. By the simplicity of $C$, $1=y^{-m}iy^{-n}$ for some $i\in I$ and some $m,n\geq 0$, whence
$y^{m+n}\in I$.
\end{proof}

\begin{lemma}
\label{Rsimlocsim} The ring $R$ is simple if
and only if $R_{\Y}$ is simple and  $v^{(m)}$ is a unit of $A$ for all
$m\geq 1$.
\end{lemma}
\begin{proof} Suppose that $R$ is simple. Then
$R_{\Y}$ is simple by \cite[2.1.16(iii)]{McCR} and its left analogue.
Suppose that, for some $m\geq 1$, $v^{(m)}$ is not a unit, and let $I=v^{(m)}A$
which,  as $v^{(m)}$ is normal, is an ideal of $A$.  Every element
of $R$ has a unique expression as a sum of elements of the form
$x^{i}a_{i,j}y^{j}$, where $a_{i,j}\in A, i\geq 0$ and $j\geq 0$.
Let $J$ be the subspace of $R$ spanned by the elements of the form
$x^iay^j$ where $i>0$ or $j\geq m$ or $a\in I$. Clearly $Jy\subseteq J$,
$JA\subseteq J$ and, as $v^{(m)}$ is not a unit, $1\notin J$. Also $x^iay^jx\in J$ if $i>0$ or, by \eqref{skewcomm}, if $j>m$.
By \eqref{skewcomm} and the $\gamma$-normality of $v^{(j)}$, \[ay^jx=\rho^{-j}(x\beta^{-1}(a)y^j-v^{(j)}\gamma(a)y^{j-1})\] for all $a\in A$.  It follows that $ay^jx\in J$ if either $j=m$, in which case $v^{(j)}\gamma(a)\in I$,
or $a\in I$. Consequently
$J$ is a proper
right ideal of $R$. Moreover $Ry^m\subset J$ so $\ann_{R}(R/J)$ is
a non-zero proper ideal of $R$, contradicting the simplicity of $R$. Hence $v^{(m)}$ is a unit
for all $m\geq 1$.

For the converse, suppose that $R_{\mathcal Y}$ is simple and that each $v^{(m)}$ is a
unit of $A$. Let $I$ be a non-zero ideal of $R$. By Lemma~\ref{yIy},
$y^{d}\in I$ for some $d\geq 0$. Choose the least such $d$. If
$d>0$, then, by \eqref{skewcomm}, \[y^{d-1} = (v^{(d)})^{-1}(xy^{d}-\rho^{d}y^{d}x)\in I,\]
contradicting the minimality of $d$. So $d=0$, $1\in I$, and
$I=R$, whence $R$ is simple.
\end{proof}

The following result is well-known; see \cite[1.8.5]{McCR} or
\cite[Theorem 1]{simskewl}.
\begin{theorem}
\label{wellk} If $\sigma$ is an automorphism of a ring $T$,
the skew Laurent polynomial ring $T[y^{\pm 1};\sigma]$
is simple if and only if $T$ is $\sigma$-simple and there does not exist $m\geq1$ such that $\sigma^m$ is
inner on $T$.
\end{theorem}

For the following variation, see
\cite[Theorem 0]{simskewl}.
\begin{theorem}
\label{notsowellk}
If $\sigma$ is an automorphism of a ring $T$,
the skew Laurent polynomial ring $T[y^{\pm 1};\sigma]$
is simple if and only if $T$ is $\sigma$-simple and there does not exist $m\geq1$ for which there is a unit $d$ of $T$ such that $\sigma(d)=d$ and $\sigma^m(a)=dad^{-1}$ for all $a\in T$.
\end{theorem}

\begin{cor}
\label{locsim} The localization $R_{\mathcal Y}$
is simple if and only if $A[w;\gamma]$ is $\alpha$-simple and $\alpha^m$ is
outer on $A[w;\gamma]$ for all $m\geq 1$.
\end{cor}
\begin{proof}
Given that $R_{\Y}=A[w;\gamma][y^{\pm 1};\alpha]$, this is immediate
from Theorem~\ref{wellk}.
\end{proof}

\begin{lemma}
\label{innerv0}
Let $m\geq 1$. If $\alpha^m$ is inner on
$A[w;\gamma]$ then $v^{(m)}=0$.
\end{lemma}
\begin{proof}
Suppose that $\alpha^{m}$ is inner on
$A[w;\gamma]$. Thus there is a unit $c=c_nw^n+\ldots+c_1w+c_0$ of $A[w;\gamma]$
such that $\alpha^{m}(w)c=cw$. By \eqref{amw}, $\rho^{-m}(w-v^{(m)})c=cw$ whence
$\gamma(c)w-v^{(m)}c=\rho^mcw$, where $\gamma$ is
extended to $A[w;\gamma]$ with $\gamma(w)=w$. Comparing constants in $A[w;\gamma]$
shows that $v^{(m)}c_0=0$. But as $c$ is a unit in $A[w;\gamma]$, $c_0$ is a unit in $A$ so $v^{(m)}=0$.
\end{proof}
\begin{rmk}
\label{vinner}
As each $v^{(m)}$ is $\gamma$-normal, if some $v^{(m)}$ is a
unit then $\gamma(a)=v^{(m)}a(v^{(m)})^{-1}$ for all $a\in A$ and every ideal of $A$ is a $\gamma$-ideal.
\end{rmk}

There are different simplicity criteria for $R$ in characteristic $0$ and characteristic $p$.  The next lemma is common to both.
\begin{lemma}\label{common}
Suppose that $A$ is $\alpha$-simple and let $J$ be a non-zero $\alpha$-ideal of
$A[w;\gamma]$. There exists a non-negative integer $m$ and an element $f=a_mw^{m}+a_{m-1} w^{m-1}+\ldots+a_1w+a_0\in J$ such that $a_m=1$ and, for $0\leq j\leq m$,
\begin{equation}
a_j=\sum_{i=0}^{m-j} \begin{pmatrix}m-i\\j\end{pmatrix}
                             \rho^{i}\alpha(a_{m-i})(-v)^{m-i-j},
\label{eqfwu}
\end{equation}
and, for all $a\in A$,
\begin{equation}
aa_{j}=a_{j}\gamma^{j-m}(a).
\label{eqgwu}
\end{equation}
\end{lemma}
\begin{proof}
For $m\geq 0$, let \[\tau_m(J)=\{0\}\cup\{a\in A:aw^m+a_{m-1}w^{m-1}+\ldots+a_iw+a_0\in J\}.\] Then $\tau_m(J)$ is an ideal of $A$ and if $a_{0},\ldots, a_{m}\in A$ then
\[\alpha\left(\sum _{i=0}^{m}a_{i}w^{i}\right)= \sum
_{i=0}^{m}\rho^{-i}\alpha(a_{i})(w-v)^{i}.\] Thus $\tau_m(J)$ is an
$\alpha$-ideal of $A$, and, by $\alpha$-simplicity, $\tau_m(J)=A$
or $\tau_m(J)=0$. Let $m\geq 0$ be minimal such that $\tau_m(J)\neq 0$.
Since $\tau_m(J)=A$, there exist $a_{0},a_1\ldots, a_{m}\in A$, with $a_m=1$, such that
$f:=a_mw^{m}+a_{m-1} w^{m-1}+\ldots+a_1w+a_0\in J$.
Then $f-\rho^{m}\alpha(f)\in J$. As $w$ and $v$ commute,
\begin{eqnarray*}
f-\rho^{m}\alpha(f) & = &\sum _{l=0} ^{m} (a_{l} w^{l} -
                           \rho^{m-l}\alpha(a_{l}) (w-v)^{l})\\
                & = &\sum_{l=0}^{m} (a_{l} w^{l} -
                           \rho^{m-l}\alpha(a_{l}))
                             \left(\sum_{j=0}^{l} {l\choose j}(-v)^{l-j}w^{j}\right)\\
                & = &\sum_{j=0}^{m} \left(a_{j} -
                             \sum_{i=0}^{m-j} {{m-i}\choose j}
                             \rho^{i}\alpha(a_{m-i})(-v)^{m-i-j}\right)w^{j}.
\end{eqnarray*}
The coefficient of $w^m$  is $0$ so, by the minimality of $m$, $f-\rho^{m}\alpha(f)=0$ and \eqref{eqfwu} follows.

For all $a\in A$, $af-f\gamma^{-m}(a)\in J$ and
\begin{equation*}
af-f\gamma^{-m}(a)=\sum_{j=0} ^{m} (aa_{j}w^{j}
-a_{j}w^{j}\gamma^{-m}(a))=\sum_{j=0} ^{m} (aa_{j}-a_{j}
                                                       \gamma^{j-m}(a))w^{j}.
\end{equation*}
The coefficient of $w^m$ here is $0$, so by the minimality of $m$, $af-f\gamma^{-m}(a)=0$ for all $a\in A$
and \eqref{eqgwu} follows.
\end{proof}

\begin{lemma}
\label{lmwsiu} Suppose that $\chr \F=0$, that $v$ is regular in $A$ and that every
$\alpha$-ideal of $A$ is a $\gamma$-ideal. Then $A[w;\gamma]$ is
$\alpha$-simple if and only if $A$ is $\alpha$-simple and the
$4$-tuple $(A,\alpha,v,\rho)$ is singular.
\end{lemma}
\begin{proof}
Suppose that $A[w;\gamma]$ is $\alpha$-simple. If $I$ is a non-zero proper $\alpha$-ideal of $A$ then, as
$I$ is also a $\gamma$-ideal of $A$,
$IA[w;\gamma]$ is a non-zero proper ideal of
$A[w;\gamma]$. It is clearly an $\alpha$-ideal so $IA[w;\gamma]=A[w;\gamma]$, whence $I=A$ and $A$
is $\alpha$-simple. If $(A,\alpha,v,\rho)$ is
conformal, with $v=u-\rho\alpha(u)$, then the Casimir element $z=w-u$ is a non-zero normal
non-unit of $A[w;\gamma]$, and $\alpha(z)=\rho^{-1}z$, so
$zA[w;\gamma]$ is a non-zero proper $\alpha$-ideal of
$A[w;\gamma]$, contradicting the $\alpha$-simplicity of $A[w;\gamma]$. Hence $(A,\alpha,v,\rho)$ is singular

Conversely, suppose that $A$ is $\alpha$-simple and that $(A,\alpha,v,\rho)$
is singular.  Let $J$ be a non-zero $\alpha$-ideal of
$A[w;\gamma]$ and let $m$ and $f$ be as in Lemma~\ref{common}. Suppose that $m\neq 0$. By \eqref{eqfwu} with $j=m-1$,
\begin{equation}\label{amminus1}
a_{m-1}-\rho\alpha(a_{m-1})=-mv.
\end{equation}
By \eqref{eqgwu} with $j=m-1$,
 $a_{m-1}a=\gamma(a)a_{m-1}$ for all $a\in A$.
Let $u=-a_{m-1}/m$. Then $ua=\gamma(a)u$ for all $a\in A$ and, by
\eqref{amminus1}, $v=u-\rho\alpha(u)$. By Lemma~\ref{guu}, $u$ is a splitting element, contradicting
the singularity of $(A,\alpha,v,\rho)$. Therefore $m=0$, $1=f\in J$, and $J=A[w;\gamma]$, which
is therefore $\alpha$-simple.
\end{proof}

We are now in a position to give a simplicity criterion for $R$ when $\F$ has characteristic $0$.
\begin{theorem}
\label{r-thm0} Suppose that $\chr \F=0$. The ring $R$ is simple
if and only if
\begin{enumerate}
\item $A$ is $\alpha$-simple;
\item the $4$-tuple $(A,\alpha,v,\rho)$ is singular;
\item for all $m\geq 1$, $v^{(m)}$ is a unit of $A$.
\end{enumerate}
\end{theorem}
\begin{proof}
Suppose that (i), (ii), (iii) hold.  By (iii) $v=v^{(1)}$ is regular and
Remark~\ref{vinner}, every $\alpha$-ideal is a $\gamma$-ideal so,
by (i), (ii) and Lemma~\ref{lmwsiu}, $A[w;\gamma]$ is
$\alpha$-simple. By (iii) and Lemma~\ref{innerv0}, each $\alpha^m$
is outer on $A[w;\gamma]$. By Lemma~\ref{Rsimlocsim} and
Corollary~\ref{locsim}, $R$ is simple.

Conversely, suppose that $R$ is simple. By Lemma~\ref{Rsimlocsim},
$R_\Y$ is simple and $v^{(m)}$ is a unit for $m\geq 1$. So (iii)
holds and, as above, $v$ is regular and every $\alpha$-ideal is a $\gamma$-ideal. By
Corollary~\ref{locsim}, $A[w;\gamma]$ is $\alpha$-simple and, by
Lemma~\ref{lmwsiu}, (i) and (ii) hold.
\end{proof}

\begin{rmk}
In the conformal case an ambiskew polynomial ring has a normal element, namely the Casimir element, and this
allows for iteration. The next result shows that iteration is also often available in the simple case. Examples of both types of iteration will appear in Section~\ref{Examples}.
\end{rmk}
\begin{theorem}\label{simpleit}
Suppose that $\chr \F=0$ and let $S=R(A,\alpha,v,\rho)$ be a simple ambiskew polynomial ring such that $v$ is an eigenvector for $\alpha$ with eigenvalue $\mu$. Let $\lambda\in \F^*$ and extend the automorphism $\alpha$ to an automorphism of $S$ by setting $\alpha(y)=\lambda y$ and $\alpha(x)=\mu\lambda^{-1} x$. The iterated ambiskew polynomial ring
$R(S,\alpha,v,\rho)$ is simple.
\end{theorem}
\begin{proof}
Using the relations $ya=\alpha(a)y$, $xa=\beta(a)x$ and
$xy-\rho yx=v$ it is routine to check that the specified extension of $\alpha$ to $S$ is indeed an automorphism,
 that $\gamma$ and $\beta$ also extend to $S$ with $\gamma(y)=\mu^{-1}y$, $\gamma(x)=\mu x$, $\beta(y)=(\lambda\mu)^{-1}y$ and $\beta(x)=\lambda x$, and that $v$ is $\gamma$-normal in $S$.  As $R$ is simple, it is $\alpha$-simple and, as $\rho$ and $v$ are unchanged, $v^{(m)}$ is
a unit for all $m\geq 1$. Finally if $R$ has a splitting element $u=\sum u_{i,j}y^ix^j$ then $u_{0,0}$ is a splitting element in $A$ and no such element exists. By Theorem~\ref{r-thm0}, $R(S,\alpha,v,\rho)$ is simple.
\end{proof}

\begin{rmk}
The situation in finite characteristic $p$ is more complex than in Theorem~\ref{r-thm0}.
We shall need to handle binomial coefficients modulo $p$ and this will be done using Lucas' Congruence \cite[p 271]{dickson} which has the consequence that if $n\geq r\in \N$ have $p$-ary representations
\[n=n_sp^s+n_{s-1}p^{s-1}+\ldots+n_1p+n_0\text{ and }r=r_sp^s+r_{s-1}p^{s-1}+\ldots+r_1p+r_0\]
where $0\leq n_i,r_i\leq p-1$ for  $s\geq i\geq 0$, then
\[{n\choose r}\not\equiv0\bmod p\Leftrightarrow n_i\geq r_i\text{ for all }i.\]

Singularity does not appear explicitly in the following analogue
of Lemma~\ref{lmwsiu} but is covered by condition (ii), with
$n=0$, in which case (a) is vacuous and (b) states that $u$ is
$\gamma$-normal and $v=\rho\alpha(u)-u$.
\end{rmk}

\begin{lemma}
\label{lmwsiup} Suppose that $\chr \F=p\neq 0$, that $v$ is regular in $A$ and that every
$\alpha$-ideal of $A$ is a $\gamma$-ideal.
  Then $A[w;\gamma]$ is $\alpha$-simple if and only if
  \begin{enumerate}
 \item
 $A$ is $\alpha$-simple;
 \item
 there do not exist a non-negative integer $n$ and elements
$u$ and $b_i$, $0\leq i\leq n-1$, of $A$
  such that \begin{enumerate}
  \item
 for $0\leq i\leq n-1$, $\gamma(b_i)=b_i$, $b_i$ is
 $\gamma^{p^n-p^i}$-normal and $\alpha(b_i)=\rho^{p^i-p^n}b_i$;
 \item
$\gamma(u)=u$, $u$ is $\gamma^{p^n}$-normal and
$\rho^{p^n}\alpha(u)-u=v^{p^n}+\sum_{0}^{n-1}b_iv^{p^i}$.
\end{enumerate}
\end{enumerate}
\end{lemma}

\begin{proof}
Suppose that $A[w;\gamma]$ is simple. As in the proof of Lemma \ref{lmwsiu}, $A$ is $\alpha$-simple.  Suppose that (ii) does not hold,
and let $h:=w^{p^n}+(\sum_{i=0}^{n-1} b_iw^{p^i})+u$ where the elements $u$ and
$b_i$, $0\leq i\leq n-1$, have the properties specified in (a) and (b). Routine calculations
show that $ha=\gamma^{p^n}(a)h$ for all $a\in A$, that $hw=wh$ and
that $\rho^{p^n}\alpha(h)=h$. Therefore $hA[w;\gamma]$ is a
non-zero proper $\alpha$-ideal of $A[w;\gamma]$, contradicting the
$\alpha$-simplicity of $A[w;\gamma]$. Thus (ii) holds.

For the converse, suppose that (i) and (ii) hold. Let $J$ be a non-zero $\alpha$-ideal of
$A[w;\gamma]$ and let $m$ and $f$ be as in Lemma~\ref{common}. Suppose that $m\neq 0$. Let $J$ be a non-zero $\alpha$-ideal of $A[w;\gamma]$ and
proceed as in the proof of Lemma~\ref{lmwsiu} to obtain a non-zero element $f=w^m+a_{m-1}w^{m-1}+\ldots+a_0$ of $J$ with minimal degree, to establish \eqref{eqfwu} and \eqref{eqgwu} and to show that
 $f-\rho^m\alpha(f)=0$ and that $af-f\gamma^{-m}a=0$ for all $a\in A$.
By \eqref{eqgwu} and Lemma~\ref{guu}, $\gamma(a_i)=a_i$ for all $i$.
As in the proof of Lemma~\ref{lmwsiu}, applying \eqref{eqfwu} and \eqref{eqgwu} with $j=m-1$ shows that $a_{m-1}-\rho\alpha(a_{m-1})=-mv$ and $a_{m-1}a=\gamma(a)a_{m-1}$ for all $a\in A$.
If $p$ does not divide $m$ then $-a_{m-1}/m$ is a splitting element, contradicting (ii) with $n=0$.
Hence $p$ divides $m$.

Express $m$ in the form $p^nr$, where $p$ does not divide $r$ and let $q=m-p^n$.
By \eqref{eqfwu} with $j=q$,
\begin{equation}
a_{q}=\sum_{i=0}^{p^n} {{m-i}\choose
q}\rho^{i}\alpha(a_{m-i})(-v)^{p^n-i}.\end{equation}
We claim that, for $0\leq i\leq p^n$, if $a_{m-i}\neq 0$ then $i=p^n$ or $i=p^n-p^t$ for some $t$
with $1\leq t\leq n$ and that $\alpha(a_{m-p^n+p^t})=\rho^{-p^t}a_{m-p^n+p^t}$ for $1\leq t\leq n$.
To see that the result will follow from this claim, suppose that this claim has been established.
By \eqref{eqfwu},
\begin{equation}\label{tempform}
a_{q}=\left(\sum_{t=1}^n {{m-p^n+p^t}\choose
m-p^n}a_{m-p^n+p^t}(-v)^{p^t}\right)+\rho^{p^n}\alpha(a_{m-p^n}).
\end{equation}
By Lucas' Congruence, $N:={m\choose{m-p^n}}={m\choose{p^n}}\not\equiv0\bmod p$.
Let $u=N^{-1}a_{m-p^n}$ and, for $0\leq i\leq n-1$,
let $b_i=N^{-1}{{m-p^n+p^i}\choose {m-p^n}}a_{m-p^n+p^i}$. For $0\leq i\leq
n-1$, $\gamma(b_i)=b_i$, $b_i$ is $\gamma^{p^n-p^i}$-normal and $\alpha(b_i)=\rho^{p^i-p^n}b_i$. Also
$\gamma(u)=u$, $u$ is $\gamma^{p^n}$-normal and, by
\eqref{tempform},
\begin{equation}
\rho^{p^n}\alpha(u)-u=v^{p^n}+ \left(\sum_{i=0}^{n-1}
b_iv^{p^i}\right).
\end{equation}
This contradicts (ii) so $m=0$ and $J=A[w;\gamma]$, which is
therefore $\alpha$-simple. It remains to establish the claim.

Suppose that the claim is false.
It holds, vacuously, when $i=p^n$ and, as $a_m=1$, it holds when $i=0$.
So there exists $i$, with $0<i<p^n$, such that
$a_{m-i}\neq 0$ and either $p^n-i\neq p^t$ for all $t$ with $1\leq t<n$ or
$i=p^n-p^t$ for some $t$ with $1\leq t<n$ and
$\alpha(a_{m-i})\neq\rho^{-i}a_{m-i}$.
For each $i$ with $0<i<p^n$, write $p^n-i=u_ip^{t_i}$, where $p$ does not divide $u_i$ and $0\leq t_i<n$. Thus if $i$ is a counterexample to the claim then either
$u_i>1$ or $u_i=1$, $i=p^n-p^{t_i}$ and $\alpha(a_{m-i})\neq\rho^{-i}a_{m-i}$. Let $t$ be the minimal
value of $t_i$ taken over all counterexamples and, among counterexamples with $t_i=t$, let $u$ be the maximal value of
$u_i$. Let $l=p^n-up^t$ be the counterexample chosen optimally in this way.

Let $0<k<l$. We shall show that ${{m-k}\choose{m-l}}a_{m-k}=0$. There are three cases.
First suppose that $t_k<t$. Then $u_k>1$ otherwise
\[k=p^n-p^{t_k}>p^n-up^t=l\]
so, by the minimality of $t$, $a_{m-k}=0$.
Next suppose that
$t_k=t$. In this case,  $u_k>u$
otherwise
\[k=p^n-u_kp^t\geq p^n-up^t=l.\]
By the maximality of $u$, $a_{m-k}=0$.
Finally, suppose that $t_k>t$. In the $p$-ary representations of  $m-k=(r-1)p^n+u_kp^{t_k}$ and $m-l=(r-1)p^n+up^{t}$, the coefficient of $p^t$ is $0$ for $m-k$ and non-zero for $m-l$. It follows from Lucas' Congruence  that
\[{{m-k}\choose{m-l}}\equiv
0\bmod p.\] This is also true if $k=0$. Combining the three cases, ${{m-k}\choose{m-l}}a_{m-k}=0={{m-k}\choose{m-l}}\alpha(a_{m-k})$ for $0\leq k<l$.
By \eqref{eqfwu}, with $j=m-l$, it follows that
$\alpha(a_{m-l})=\rho^{-l}a_{m-l}$.
As $l$ is a counterexample to
the claim and $\alpha(a_{m-l})=\rho^{-l}a_{m-l}$, it must be the case that $u=u_l>1$.

Now let $s=l+p^t$ and let $j=m-s=m-p^n+(u-1)p^t$. Let $0\leq i\leq s$ be such that the $i$th term on the right hand side of \eqref{eqfwu}, is non-zero. Thus $a_{m-i}\neq 0$ and ${{m-i}\choose{j}}\not\equiv 0\bmod p$. We shall see that $i=l$ or $i=s$.
By Lucas' Congruence,
\[{{m-i}\choose{j}}={{rp^n}\choose{(r-1)p^n+(u-1)p^t}}\equiv
0\bmod p\] so we can assume that $i>0$.
Now suppose that
$t_i>t$. Then, by Lucas' Congruence,
\[{{m-i}\choose{j}}={{(r-1)p^n+u_ip^{t_i}}\choose{(r-1)p^n+(u-1)p^{t}}}\equiv
0\bmod p.\] So $t_i\leq t$.
Suppose that $t_i<t$. As $a_{m-i}\neq 0$ it follows, by minimality of $t$,  that $u_i=1$ and hence  \[i=p^n-p^{t_i}>p^n-p^{t}\geq p^n-(u-1)p^{t}=p^-(p^n-l-p^t)=s,\]
which is false. So $t_i=t$.
As $a_{m-i}\neq0$ either $i=l$ or $u_i=1$ in which
case, as above, $i=p^n-p^{t}\geq s$,
or  $i=p^n-u_ip^t$, where $u_i\leq u$ by the maximality of $u$. In the final possibility,
$i=p^n-u_ip^t\geq p^n-(u-1)p^t=l+p^t=s$.   Thus $i=l$ or $i=s$.
   By \eqref{eqfwu}, we now know that
\begin{equation}
\label{secondaj}
a_{j}={{m-l}\choose{m-j}}\rho^{l}\alpha(a_{m-l})(-v)^{p^t}+\rho^{m-j}\alpha(a_{j}).
\end{equation}
Now $a_{m-l}$ is a unit, by $\alpha$-simplicity and \eqref{eqgwu}, and $\alpha(a_{m-l})=\rho^{-l}a_{m-l}$ so
$\alpha(a_{m-l}^{-1})=\rho^{l}a_{m-l}^{-1}$. Also
$M:={{m-l}\choose{m-j}}={{p^nr-p^n+up^t}\choose{p^t}}\not\equiv
0\bmod p$ by Lucas' Congruence. Multiplying throughout \eqref{secondaj} on the left by
$M^{-1}a_{m-l}^{-1}$,
gives $v^{p^t}=u-\rho^{p^t}\alpha(u)$, where
$u=M^{-1}a_{m-l}^{-1}a_{j}$, so that $\rho^{p^t}\alpha(u)=M^{-1}\rho^{m-j}a_{m-l}^{-1}\alpha(a_{j})$. In view of \eqref{eqgwu}, and as
$\gamma(a_i)=a_i$ for all $i$, $u$ is $\gamma^{p^t}$-normal.
Setting $n=t$ and $b_i=0$ for $0\leq i\leq n-1$, this contradicts
(ii) so the claim is established and the result follows.
\end{proof}

\begin{theorem}
\label{r-thmp} If $\chr \F=p\neq 0$, then
$R(A,\alpha,v,\rho)$ is simple if and only if
  \begin{enumerate}
 \item
 $A$ is $\alpha$-simple;
 \item
 there do not exist a non-negative integer $n$ and elements
$b_i$, $n-1\geq i\geq 0$, and $u$ of $A$
  such that
  \begin{enumerate}
  \item
 for $0\leq i\leq n-1$, $\gamma(b_i)=b_i$, $b_i$ is
 $\gamma^{p^n-p^i}$-normal and $\alpha(b_i)=\rho^{p^i-p^n}b_i$;
 \item
$\gamma(u)=u$, $u$ is $\gamma^{p^n}$-normal and
$\rho^{p^n}\alpha(u)-u=v^{p^n}+\sum_{0}^{n-1}b_iv^{p^i}$;
\end{enumerate}
\item
for all $m\geq 1$, $v^{(m)}$ is a unit of $A$.\end{enumerate}
\end{theorem}
\begin{proof}
The proof is similar to that of Theorem~\ref{r-thm0}, with
Lemma~\ref{lmwsiup} replacing Lemma~\ref{lmwsiu}.
\end{proof}

\section{Simple localizations}\label{simpleS}
In this section, we suppose that the $4$-tuple
$(A,\alpha,v,\rho)$ is conformal with splitting element $u$. Thus $R$ is not simple because the Casimir
element $z=xy-u$ is normal in $R$ and generates a proper non-zero
ideal. This obstruction to simplicity can be removed by localizing
at $\mathcal{Z}:=\{z^i\}_{i\geq 0}$ or by factoring out $zR$. In this
section we shall give simplicity criteria for the localization
$S:=R_\mathcal{Z}$ and, as
with $R$, approach the problem through the localization $S_\Y$. As
$zy=\rho yz$, $\Y$ remains a right and left Ore set in $S$ and $S_\Y$ can be
identified with $A[z^{\pm 1};\gamma][y^{\pm 1};\alpha]$, where
$\alpha(z)=\rho^{-1}z$.
\begin{lemma}
\label{Ssimlocsim} The ring $S$ is simple if and only if $S_\Y$ is
simple and, for all $m\geq 1$, there exists $n\geq 0$ such that $u^n\in
v^{(m)}A$.
\end{lemma}
\begin{proof}
Suppose that $S$ is simple. Then $S_\Y$ is simple by \cite[Proposition 10.17(a)]{GW}. Let $m\geq 1$.
As in the proof of
Lemma~\ref{Rsimlocsim}, let $J$ be the $\F$-subspace of $R$ spanned by the elements of the form
$x^iay^j$ where $i>0$ or $j\geq m$ or $a\in v^{(m)}A$. Then
$J$ is a
right ideal of $R$ and $I:=\ann_{R}(R/J)$ is
a non-zero proper ideal of $R$ contained in $J$. Note that
$J\cap A=v^{(m)}A$. As $S=R_\mathcal{Z}$ is simple, $z^n\in I$ for some $n\geq
0$ by Lemma~\ref{yIy}. Thus $(xy-u)^n\in J$ and, as $x\in J$ and $u^rxy=x\alpha(u^r)y\in
J$ for $1\leq r<n$, it follows that $u^n\in J\cap A=v^{(m)}A$.

Conversely, suppose that $S_\Y$ is simple and that, for all $m\geq
1$, there exists $n\geq 0$ such that $u^n\in v^{(m)}A$. Let $I$ be a non-zero ideal
of $S$. By Lemma~\ref{yIy}, $y^m\in I$ for some $m\geq 0$. Choose the least such
$m$ and suppose that $m\neq 0$. There exists $n\geq 0$ such that
$u^n\in v^{(m)}A$. By \eqref{skewcomm}, $v^{(m)}y^{m-1}\in I$ and, by the normality of $v^{(m)}$,
$v^{(m)}Ay^{m-1}\subseteq I$ so $u^ny^{m-1}\in I$. Choose $t\geq 0$ minimal
such that $u^ty^{m-1}\in I$ and suppose that $t\neq 0$. Recall that $uz=zu$ so
$zu^{t-1}y^{m-1}=u^{t-1}zy^{m-1}=u^{t-1}xy^m-u^ty^{m-1}\in I$. As $z$
is invertible in $S$, it follows that $u^{t-1}y^{m-1}\in I$,
contradicting the choice of $t$. Thus $t=0$ and $y^{m-1}\in I$,
contradicting the choice of $m$. Therefore $m=0$, $1\in I$ and $S$
is simple.
\end{proof}

\begin{lemma}
\label{innerv0z} Let $m\geq 1$. If $\alpha^m$ is inner on
$A[z^{\pm 1};\gamma]$ then $v^{(m)}=0$.
\end{lemma}
\begin{proof}
The element $uz^{-1}$ is central in $A[z^{\pm 1};\gamma]$ so if
$\alpha^{m}$ is inner then $\alpha^m(uz^{-1})=uz^{-1}$. As
$\alpha^m(z)=\rho^{-m}z$ it follows that $\alpha^m(u)=\rho^{-m}u$
and hence that $v^{(m)}=u-\rho^m\alpha^m(u)=0$.
\end{proof}
\begin{defn}
Let $m,j\in \Z$. A non-zero element $c$ of $A$ will be called
$(m,j)$-{\it special} if $\gamma(c)=\rho^{m}c,
\alpha(c)=\rho^jc$ and $c$ is $\alpha^m\gamma^{-j}$-normal. Thus $c\gamma^j(a)=\alpha^m(a)c$ for all $a\in A$. If $A$ is $\{\alpha,\gamma\}$-simple
then any
$(m,j)$-special element of $A$ is necessarily a unit in $A$ and its inverse is
$(-m,-j)$-special.
\end{defn}
\begin{rmk}\label{rhoone}
If $\rho$ is a root of unity and $\gamma=\id_A$ then $1$ is
$(0,j)$-special for all $j$ such that $\rho^j=1$.
\end{rmk}

\begin{lemma} \label{aminner} Suppose that $A$ is
$\{\alpha,\gamma\}$-simple and let $m\geq 1$. Then $\alpha^m$ is
inner on $A[z^{\pm 1};\gamma]$, induced by a unit $d$ such that $\alpha(d)=d$, if and only if there exists $j\in\Z$
such that $A$ has an  $(m,j)$-special element.
\end{lemma}
\begin{proof}
First suppose that $c\in A$ is $(m,j)$-special. By $\{\alpha,\gamma\}$-simplicity, $c\in U(A)$. Let $d=cz^j$. Thus
$d\in U(A[z^{\pm 1};\gamma])$ and  $\alpha(d)=\rho^jc\rho^{-j}z^j=d$. Also
\[dz=cz^{j+1}=z\gamma^{-1}(c)z^j=\rho^{-m}zcz^{j}=\alpha^{m}(z)d\]
and, for $a\in A$,
\[da=c\gamma^{j}(a)z^j=\alpha^m(a)cz^j=\alpha^m(a)d.\] Thus $\alpha(d)=d$ and $\alpha^m$
is the inner automorphism of $A[z^{\pm 1};\gamma]$ induced by $d$.

Conversely, suppose that $\alpha^m$ is inner on
$A[z^{\pm 1};\gamma]$ induced by a unit $d=\sum_{i=t}^j
c_iz^i\in A[z^{\pm 1};\gamma]$ such that $\alpha(d)=d$ and $c_j\neq 0$. (If $A$ is a domain then $t=j$.)
Then $df=\alpha^m(f)d$ for all
$f\in A[z^{\pm 1};\gamma]$ and  $\alpha(c_j)=\rho^jc_j$. As \[\left(\sum c_iz^i\right)z=\alpha^m(z)\left(\sum
c_iz^i\right)=\rho^{-m}\left(\sum \gamma(c_i)z^i\right)z,\] we also have
$\gamma(c_j)=\rho^m c_j$. Finally, for $a\in
A$, \[\alpha^m(a)\left(\sum c_i z^i\right)=\left(\sum c_i z^i\right)a=\sum
c_i\gamma^i(a)z^i,\] so $c_j\gamma^j(a)=\alpha^m(a)c_j$.  Thus
$c_j$ is $\{m,j\}$-special.
\end{proof}

\begin{lemma} \label{lmzsic} The ring $A[z^{\pm 1};\gamma]$ is
$\alpha$-simple if and only if $A$ is $\{\alpha,\gamma\}$-simple
and there does not exist  $j\geq 1$ such that $A$ has a
$(0,-j)$-special element.
\end{lemma}
\begin{proof}
Suppose that $A[z^{\pm 1};\gamma]$ is $\alpha$-simple. If $I$ is a
non-zero proper $\{\alpha,\gamma\}$-ideal of $A$ then $IA[z^{\pm 1};\gamma]$ is a non-zero proper
ideal of $A[z^{\pm 1};\gamma]$ and, as
$\alpha(z)=\rho^{-1}z$, it is an $\alpha$-ideal.  Hence $A$ is
$\{\alpha,\gamma\}$-simple. Let $j\geq 1$ be such that $A$ has a
$(0,-j)$-special element $c$. Then
\[(z^{j}-c)z=z(z^{j}-c),\;
\alpha(z^{j}-c)=\rho^{-j}(z^{j}-c)\text{ and }
(z^{j}-c)a=\gamma^{j}(a)(z^{j}-c)\]
for all $a\in A$, whence
$(z^{j}-c)A[z^{\pm 1};\gamma]$ is a non-zero proper $\alpha$-ideal
of $A[z^{\pm 1};\gamma]$. Thus no such $j$ exists.

Conversely, suppose that $A$ is $\{\alpha,\gamma\}$-simple, and
that, for $j\geq 1$, $A$ has no $(0,-j)$-special element. Let
$J$ be a non-zero $\alpha$-ideal of $A[z^{\pm 1};\gamma]$. For each
$m\geq 0$, let $L_{m}$ denote the ideal of $A$ consisting of the
leading coefficients of those elements of $J\cap A[z;\gamma]$
which have degree $m$, together with 0.  For
$a_{0},\ldots,a_{m}\in A$,
\[\rho^{m}\alpha\left(\sum_{i=0}^{m}a_{i}z^{i}\right)=\sum_{i=0}^{m}\rho^{m-i}\alpha(a_{i})z^{i}\,\text{ and }\,z\left(\sum_{i=0}^{m}a_{i}z^{i}\right)z^{-1}=\sum_{i=0}^{j}\gamma(a_{i})z^{i},\]
whence  $L_{m}$ is an $\{\alpha,\gamma\}$-ideal of $A$.  Choose
 $j$ minimal such that $L_{j}\neq 0$, whence $L_j=A$, and suppose that $j\neq 0$. There exist
$c_{0},\ldots,c_{j}\in A$, with $c_{j}=1$  such that $g:=\sum_{i=0}^{j}c_{i}z^{i}\in
J\cap A[z;\gamma]$. By the minimality of
$j$,  $c_{0}\neq 0$. Then
\[g-zgz^{-1}= \sum_{i=0}^{j-1}(c_{i}-\gamma(c_i))z^{i}\in J\cap A[z;\gamma],\] whence $\gamma(c_0)=c_0$. Also
\[g-\rho^{j}\alpha(g)=\sum_{i=0}^{j-1}(c_{i}-\rho^{j-i}\alpha(c_{i}))z^{i}\in J\cap A[z;\gamma],\] whence
$\rho^{j}\alpha(c_0)=c_0$. Finally,
\[bg-g\gamma^{-j}(b)=\sum_{i=0}^{j-1}(bc_{i}-c_{i}\gamma^{i-j}(b))z^{i}\in J\cap A[z;\gamma]\] for all $b\in A$,
whence $c_0$ is $\gamma^j$-normal. Thus
$c_0$ is $(0,-j)$-special. This is a contradiction so
$j=0$, $1=g\in J$ and $J=A[z^{\pm 1};\gamma]$, which is therefore
$\alpha$-simple.
\end{proof}

\begin{theorem}
\label{s-thm} The ring $S=R_\mathcal{Z}$ is simple if and only if the following hold:
\begin{enumerate}
\item $A$ is $\{\alpha,\gamma\}$-simple;
\item there do not exist $c\in A$ and $m,j\in\Z$, with $m$ and $j$ not both $0$, such that $c$ is
$(m,j)$-special;
\item  for all
$m\geq 1$, there exists $n\geq 0$ such that $u^n\in v^{(m)}A$.
\end{enumerate}
\end{theorem}
\begin{proof}
Suppose that $S$ is simple. By Lemma \ref{Ssimlocsim}, (iii) holds and
$S_\mathcal{Y}$ is simple.
By Theorem~\ref{notsowellk}, $A[z^{\pm 1};\gamma]$ is $\alpha$-simple and there does not exist $m\geq 1$ for
which $\alpha^{m}$ is inner on $A[z^{\pm 1};\gamma]$, induced by a unit $d$ such that $\alpha(d)=d$.
By Lemma~\ref{lmzsic}, $A$ is $\{\alpha,\gamma\}$-simple so (i) holds. Finally, suppose that $c\in A$ and that $m,j\in\Z$
are such that $c$ is
$(m,j)$-special. By Lemma~\ref{aminner}, we cannot have $m>0$ and, by the symmetry arising from $\{\alpha,\gamma\}$-simplicity, we cannot have $m<0$. Thus $m=0$. By Lemma~\ref{lmzsic} we cannot have $j<0$ and, by  symmetry, we cannot have $j>0$. Thus $m=j=0$ and
(ii) holds.

Conversely suppose that (i), (ii) and (iii) hold. By Lemma~\ref{aminner}, there does not exist $m\geq 1$ for
which $\alpha^{m}$ is inner on $A[z^{\pm 1};\gamma]$, induced by a unit $d$ such that $\alpha(d)=d$. By  Lemma~\ref{lmzsic}, $A[z^{\pm 1};\gamma]$ is $\alpha$-simple. It follows, by Theorem~\ref{notsowellk}, that $S_\mathcal{Y}$ is simple and, by Lemma~\ref{Ssimlocsim} and (iii), that $S$ is simple.
\end{proof}

\begin{rmk}
When $u$ is nilpotent, statement (iii) of Theorem~\ref{s-thm} holds vacuously
so the theorem simplifies. In particular, this applies when $u=0$, in which case $v=0$, $z=xy=\rho yx$, $y^{-1}\in S$ and
$S=S_\Y$ is an iterated skew Laurent extension $A[z^{\pm 1};\gamma][y^{\pm 1};\alpha]$.  When $u$ is not nilpotent, statement (ii) can be relaxed, as follows.
\end{rmk}

\begin{theorem}
\label{s-thm2} Suppose that $u$ is not nilpotent. The ring
$S$ is simple if and only if
\begin{enumerate}
\item $A$ is $\{\alpha,\gamma\}$-simple;
\item there does not exist  $j\in \Z\backslash\{0\}$ such that $A$ has a
$(0,j)$-special element;
\item  for all
$m\geq 1$, there exists $n\geq 0$ such that $u^n\in v^{(m)}A$.
\end{enumerate}
\end{theorem}
\begin{proof} Suppose that (i) and (iii) hold.
Suppose that there exists $m\neq0$ such that $A$ has a normal
$(m,j)$-special element $c$ for some $j\in \Z$. By (i) $c$ is a unit so, replacing $c$ by $c^{-1}$ if $m<0$, we can assume that $m\geq 1$. By Lemma \ref{aminner}, $\alpha^m$ is inner and, by Lemma \ref{innerv0z},
$v^{(m)}=0$, contradicting (iii). Therefore statement (ii) of Theorem~\ref{s-thm} can be relaxed as stated.
\end{proof}

\begin{rmk}\label{rhoone1}
Recall from Remark~\ref{rhoone} that if $\rho$ is a root of unity and $\gamma=\id_A$ then $1$ is
$(0,j)$-special for all $j\in \Z$ such that $\rho^j=1$ and consequently $S$ is not simple. In this situation $z^j-\lambda$
is central for all $\lambda\in \F$.
\end{rmk}

\section{Generalized Weyl algebras}\label{simpleT}
In this section, we again suppose that the $4$-tuple
$(A,\alpha,v,\rho)$ is conformal with splitting element $u$. Here we consider the simplicity of $T:=R/zR$ and let $X=x+zR$ and $Y=y+zR$. As $zR\cap
A=0$, $A$ embeds in $T$ and we can identify each element $a$ of
$A$  with $a+zR$. Then $T$ is the ring extension of $A$ generated
by $X$ and $Y$ subject to the relations $XY=u$, $YX=\alpha(u)$,
and, for all $a\in A$, $Xa=\beta(a)X$ and $Ya=\alpha(a)Y$. Note that $T$, which we occasionally denote by $T(A,\alpha,u)$, depends on $u$ rather than on $v$ and $\rho$. In the
case where $u$ is central this ring was studied by Rosenberg in \cite{rosen}, under the name
{\it hyperbolic ring}, by the author in
\cite{primitive}, without a name,  and by Bavula in a sequence of papers, including \cite{fa,vlad1,vlad5},
under the name {\it generalized Weyl algebra}. We shall extend
use of the latter name to the case where $u$ is normal and not
necessarily central.
\begin{lemma}\label{alphacomm}
For all $i\in\Z$, $\alpha^i(u)$ is $\gamma$-normal in $A$ and $\alpha^i(u)\alpha^j(u)=\alpha^j(u)\alpha^i(u)$ for all $j\in\Z$.
\end{lemma}
\begin{proof} As $\alpha\gamma=\gamma\alpha$, it is readily checked that $\alpha^i(u)$ is $\gamma$-normal in $A$ for all $i$
and hence that
\[\alpha^i(u)\alpha^j(u)=\gamma(\alpha^j(u))\alpha^i(u)=\alpha^j(\gamma(u))\alpha^i(u)=\alpha^j(u)\alpha^i(u).\]
\end{proof}

\begin{rmk}
\label{Trmks}
Here we list some basic properties of generalized Weyl
algebras. These are well-documented in the above papers, in
particular \cite{primitive} where the notation is close to that
used here, and remain true in the present generality.

\noindent (i)
For all $m\geq 1$,
\[X^{m}Y^{m}=\displaystyle\prod_{i=0}^{m-1}\alpha^{-i}(u)\text{ and }
Y^{m}X^{m}=\displaystyle\prod_{i=1}^{m}  \alpha^{ i}(u).\] By Lemma~\ref{alphacomm}, these
products are independent of the order of the terms.

\noindent (ii)
Every element of $T$ can be written uniquely in each of the forms
\[
a+\sum_{i=1}^{m}b_{i}Y^{i}+\sum_{j=1}^{n}c_{j}X^{j}=a+\sum_{i=1}^{m}Y^{i}b^\prime_{i}+\sum_{j=1}^{n}X^{j}c^\prime_{j}, \]
where $a,b_{i},c_{j},b^\prime_{i},c^\prime_{j}\in A$, $1\leq i\leq m$, $1\leq j\leq n$.

\noindent (iii) If $A$ is a domain then $T$ is a domain.

\noindent (iv)
The ring $T$ is $\Z$-graded with $\deg Y=1$, $\deg X=-1$, and $T_0=A$.  For $m\geq 1$, $T_m=AY^m=Y^mA$ and $T_{-m}=AX^m=X^mA$.
\end{rmk}

\begin{notn} For the remainder of this section, let ${\Y}=\{Y^{i}:i\in\N\}$ and
${\X}=\{X^{i}:i\in\N\}$, the images in $T$ of the right and left Ore subsets of $R$ for which the notation $\Y$ and $\X$ was used previously. Thus $\Y$ and $\X$ are right and left
Ore sets in $T$. If $u$ is regular in $A$ then $Y$ and $X$ are regular in
$T$ and the localizations $T_\Y$ and $T_\X$ are isomorphic to
$A[Y^{\pm 1};\alpha]$ and $A[X^{\pm 1};\beta]$ respectively, with
$X\mapsto uY^{-1}$ in the former and $Y\mapsto \alpha(u)X^{-1}$ in
the latter.
\end{notn}

The following result generalizes \cite[Theorem 6.1]{primitive},
where $A$ is commutative. Bavula \cite{vlad5} gives a proof for
the case where $u$ is central.

\begin{theorem} Let $\alpha$ and $\gamma$ be commuting $\F$-automorphisms $A$ of an $\F$-algebra $A$, let $u\in A$ be $\gamma$-normal  and let $T$ be the generalized Weyl algebra $T(A,\alpha,u)$.
\label{t-thm} Then $T$ is simple if and only if
\begin{enumerate}
\item  $A$ is $\alpha$-simple;
\item   $\alpha^{m}$ is outer for all $m\geq 1$;
\item  $u$ is regular;
\item   $uA+\alpha^{m}(u)A=A$ for all $m\geq 1$.
\end{enumerate}
\end{theorem}
\begin{proof}
Suppose that $T$ is simple. As $\{t\in T:tY^n=0\text{ for some }n\geq 1\}$ and $\{t\in T:Y^nt=0\text{ for some }n\geq 1\}$ are proper ideals of $T$, $Y$ is regular in $T$. Similarly $X$ is regular in $T$. If $a\in A$ is such that if $au=0$ then $aXY=0$, whence $a=0$.
On the other side, if  $ua=0$ then  $\alpha(u)\alpha(a)=0$, $YX\alpha(a)=0$, $\alpha(a)=0$ and $a=0$. Hence $u$ is
regular in $A$ and $T_\Y\simeq A[Y^{\pm 1};\alpha]$, which is
simple by \cite[Proposition 10.17(a)]{GW}. By Theorem~\ref{wellk}, $A$ is $\alpha$-simple and
$\alpha^m$ is outer for all $m\geq 1$. Thus (i), (ii) and (iii) hold.

Suppose that (iv) does not hold. Then
there exist $m\geq 1$ and a maximal right ideal $M$ of $A$ with
$u,\alpha^{m}(u)\in M$. Let $I$ denote the right ideal
$MT+XT+Y^{m}T$. Let $t\in I$. Then $XY=u\in M$ and, by Remark ~\ref{Trmks}(i), $Y^mX^m\in M$ so, in the $\Z$-grading on $T$, $t_0\in M+XYA+Y^mX^mA\subseteq M$. Thus $I$ is
proper and $\ann_{T}(T/I)$ is a proper ideal of $T$. Also, $X^iAY^m\subseteq I$ for $i\geq 1$  and $Y^iAY^m\subseteq I$ for $i\geq 0$,
so $TY^m\subseteq I$, $Y^{m}\in\ann_{T}(T/I)$ and $\ann_{T}(T/I)\neq 0$, contradicting the
simplicity of $T$. Thus (i)--(iv) hold.

Conversely, suppose that (i)--(iv) hold. As $u$ is regular in $A$, $Y$
is regular in $T$ and $T_{\mathcal Y}\cong A[Y^{\pm 1};\alpha]$. By (i),
(ii) and Theorem~\ref{wellk}, $T_{\mathcal Y}$ is simple.  Suppose
that $T$ has a non-zero proper ideal $J$. By
Lemma~\ref{yIy},
$Y^n\in J$ for some $n\geq 1$.  Then $X^nY^n\in J$ and $Y^nX^n\in
J$, that is, by Remark ~\ref{Trmks}(i),
$u\alpha^{-1}(u)\ldots\alpha^{1-n}(u)\in J$ and
$\alpha(u)\alpha^{2}(u)\ldots\alpha^n(u)\in J$. Let $N$ be a
maximal ideal of $A$ containing $J\cap A$.  By Lemma
~\ref{alphacomm}, each $\alpha^{i}(u)$ is normal in $A$ so there
are integers $s$ and $t$ such that $s<t$,
$\alpha^s(u)\in N$ and $\alpha^t(u)\in N.$ Setting $m=t-s>0$ and
$M=\alpha^{-s}(N)$, we see that $M\supseteq uA+\alpha^m(u)A=A$,
contradicting (iv). Therefore $T$ is simple.
\end{proof}

\begin{rmk}
Every ambiskew polynomial ring $R(A,\alpha,v,\rho]$ is isomorphic
to a generalized Weyl algebra over $A[w;\gamma]$, where $w$ and $\gamma$ are as in Notation \ref{w}. Thus
  $xy=w$ and $yx=\rho^{-1}(w-v)$. See \cite[2.6 Corollary]{jw} for the case where $A$ is commutative.
This gives an alternative approach to the simplicity criteria in Section 3 by means of the above Theorem and the passage, via Lemmas~\ref{lmwsiu} and \ref{lmwsiup}, between $A[w;\gamma]$ and $A$.
\end{rmk}

\section{Examples}\label{Examples}

\begin{example}[{\bf Ambiskew polynomial rings over $\F$}]\label{fromk}
If, in the definition of ambiskew polynomial ring in \ref{defambi}, we take $A=\F$ then we obtain an $\F$-algebra $R$ generated by $x$ and $y$ subject to the relation
$xy-\rho yx=v$, for some $v,\rho\in \F$ with $\rho\neq 0$. (The case $\rho=0$ can be included by relaxing the condition that $\alpha$ should be an automorphism, see \cite{ambi}.) If $v\neq 0$ then, replacing $x$ by $v^{-1}x$, we can assume
 that $v=1$. To fit with existing literature, we write $q$ for $\rho$. Apart from the commutative polynomial algebra, there are essentially three cases:
 \[{\rm (i)}\; xy=qyx, q\neq 1;\quad {\rm (ii)}\; xy-yx=1;\quad {\rm (iii)}\; xy-qyx=1, q\neq 1.\]
 These correspond to the {\it co-ordinate ring of the quantum plane}, the {\it first Weyl algebra} $A_1(\F)$, and the
 {\it quantized Weyl algebra} $A_1^q(\F)$ respectively. Simplicity issues for these algebras are well-understood, for example see \cite[Corollary 1.18]{GW} or \cite[Example 1.8.7(ii) with $n=1$]{McCR} for (i), \cite[Corollary 2.2 and the passage following it]{GW} for (ii)
 and \cite[8.4 and 8.5]{g} for (iii).
 However
 they illustrate Theorems~\ref{r-thm0}, \ref{r-thmp}, \ref{s-thm} and \ref{t-thm} nicely.

 In (i), $R$ is conformal with splitting element $0$ and Casimir element $z=xy$ so, by Theorems~\ref{r-thm0}, \ref{r-thmp} and \ref{t-thm}, $R$ and $R/zR$ are not simple. Also statements (i) and (iii) of Theorem~\ref{s-thm} are true. The localization $S=R_\mathcal{Z}$ is the quantized co-ordinate ring $\mathcal{O}_q(T_2)$ of the $2$-torus, that is $\F[x^{\pm 1},y^{\pm 1}; xy=qyx]$. For $0\neq c\in \F$, $\alpha(c)=c=\gamma(c)$
 so $c$ cannot be $(m,j)$-special if $m\neq0\neq j$ and $q$ is not a root of unity. If $q$ is an $m$th root of unity then $1$ is $(m,m)$-special. By Theorem \ref{s-thm}, $S$ is simple if and only if $q$ is not a root of unity.

 In (iii), $R$ is again conformal, with splitting element $(1-q)^{-1}$ and Casimir element $z:=xy-(1-q)^{-1}$, and $\alpha=\id_\F$ is inner. By Theorems~\ref{r-thm0}, \ref{r-thmp} and \ref{t-thm}, $R$ and $R/zR$ are not simple. As in (i),
 there is an $(m,j)$-special element with $m\neq 0\neq j$ if and only if $q$ is a root of unity. Also $u=(1-q)^{-1}$ and $v^{(m)}=\qnum{m}{q}$, where, for $m\in \N$ and $q\in \F^*$, $\qnum{m}{q}:=1+q+q^2+\ldots+q^{m-1}$, which, if $q\neq 1$, is $\dfrac{q^m-1}{q-1}$. Thus if $q$ is not a root of unity then $v^{(m)}$ is a unit for all $m\geq 1$. By Theorem \ref{s-thm}, $R_\mathcal{Z}$ is simple if and only if $q$ is not a root of unity.

 In (ii), where $\rho=1=v$, $v^{(m)}=m$ for $m\geq 1$, and $\alpha=\id_\F$, $u-\rho\alpha(u)=0\neq v$ for all $u\in \F$ so $R$ is singular. If $\chr \F=0$ then $v$ is a unit so $R$ is simple by Theorem~\ref{r-thm0}.
In characteristic $p>0$,  both conditions (ii) and (iii) of Theorem~\ref{r-thmp} fail. To see the failure of
 (ii), take $a=1$, $n=1$ and $b_0=-1$, so that the element $h$ in the proof of Lemma~\ref{lmwsiup} is $w^p$, and for (iii), note that $v^{(p)}=0$.
\end{example}

\begin{example}
Consider $\C$ as an $\R$-algebra and let $\alpha:\C\rightarrow \C$ be the $\R$-automorphism of $\C$ given by complex conjugation. Thus we can form $R=R(\C,\alpha,a+ib,\rho)$ for any $a,b,\rho\in \R$ with $\rho\neq 0$. As an $\R$-algebra, $R$ is generated by $i,x$ and $y$ subject to the relations
\[i^2=-1,\quad xi=-ix,\quad yi=-iy,\quad xy-\rho yx=a+ib.\]
It is readily checked that $R$ is singular if and only if either $\rho=1$ and $a\neq 0$ or $\rho=-1$ and $b\neq 0$. Also,
for $m\geq 1$,
\[v^{(m)}=\begin{cases}
am\text{ if }m\text{ is even}\text{ and } \rho=1,\\
am+ib\text{ if }m\text{ is odd}\text{ and } \rho=1,\\
ibm\text{ if }m\text{ is even}\text{ and } \rho=-1,\\
a+ibm\text{ if }m\text{ is odd}\text{ and } \rho=-1.
\end{cases}\]
As $\C$ is $\alpha$-simple, it follows from Theorem~\ref{r-thm0} that $R$ is simple if and only if
either $\rho=1$ and $a\neq 0$ or $\rho=-1$ and $b\neq 0$. In particular $R(\C,\alpha,1,1)$, in which
$xy-yx=1$, and $R(\C,\alpha,i,-1)$, in which
$xy+yx=i$, are simple.

Similar examples can be constructed for any quadratic field extension $K$ of a field $\F$ with $\chr \F=0$,
for example $K=\Q(\sqrt{d})$ and $\F=\Q$ with $d\in \Z$ square-free.

\end{example}

\begin{example}[{\bf Quantum tori}]
Let $n$ be a positive integer and, for $1\leq i,j\leq n$, let $Q=(q_{i,j})$ be an $n\times n$ matrix over $\F$
such that $q_{i,j}\neq 0$ for all $i,j$, $q_{i,i}=1$ for all $i$ and $q_{j,i}=q_{i,j}^{-1}$ whenever
$j\neq i$. The {\it quantized co-ordinate ring} $\mathcal{O}_Q(T_n)$ {\it of the} $n$-{\it torus} is the $\F$-algebra
generated by $x_1, x_2,\ldots,x_n$ and their inverses subject to the relations
$x_ix_j=q_{i,j}x_jx_i$ for $1\leq i<j\leq n$. When $n=2$ this is the algebra $\mathcal{O}_q(T_2)$ from \ref{fromk}, with $q=q_{1,2}$. Thus it is the localization of a conformal ambiskew polynomial ring at the powers of the Casimir element. To obtain
$\mathcal{O}_Q(T_3)$ in a similar way, take $A=\F[x_1^{\pm 1}]$, $\alpha(x_1)=q_{3,1}x_1$, $\gamma(x_1)=q_{2,1}q_{3,1}x_1$, $\beta(x_1)=q_{2,1}x_1$, $v=0$ and $\rho=q_{2,3}$, setting $y=x_3, x=x_2$ and inverting $xy$. For $n\geq 4$, $\mathcal{O}_Q(T_n)$ can be constructed, inductively, by forming $R(\mathcal{O}_Q(T_{n-2}),\alpha,0,q_{n-1,n})$, with $\alpha(x_i)=q_{n,i}x_i$, $\gamma(x_i)=q_{n-1,i}q_{n,i}x_i$,
$\beta(x_i)=q_{n-1,i}x_i$, $y=x_n$ and $x=x_{n-1}$ and inverting the Casimir element $xy$.

McConnell and Pettit \cite[Proposition 1.3]{McCP} have shown that $\mathcal{O}_Q(T_n)$ is simple if and only if $Q$ has the property ($*$) that if
$m_1,m_2,\ldots,m_n\in \Z$ are such that
$q_{1,i}^{m_1}q_{2,i}^{m_2}\ldots q_{n,i}^{m_n}=1$ for $1\leq i\leq n$, then $m_1=m_2=\ldots=m_n=0$.  Although Theorem~\ref{s-thm} does not give a quicker method of proof of this simplicity criterion, it is instructive to interpret the conditions of that theorem in this example. As $u=0$, the statement in Theorem~\ref{s-thm}(iii) is always true here.  Statement (i) of Theorem~\ref{s-thm}, that $\mathcal{O}_Q(T_{n-2})$ is $\{\alpha,\gamma\}$-simple, is equivalent to the special case of  ($*$) where $m_{n-1}=0=m_n$, that is, if $m_1,m_2,\ldots,m_{n-2}\in \Z$ are such that
$q_{1,i}^{m_1}q_{2,i}^{m_2}\ldots q_{n-2,i}^{m_{n-2}}=1$ for $1\leq j\leq n$ then $m_1=m_2=\ldots=m_{n-2}=0$.
If this holds then any $(m,j)$-special element $c$ must be invertible and therefore of the form $\mu x_1^{m_1}x_2^{m_2}\ldots x_{n-2}^{m_{n-2}}$ for some
$\mu \in \F^*$ and some $m_1, m_2,\ldots,m_{n-2}\in \Z.$ Such an element is $(m,j)$-special if and only if
 \[q_{1,i}^{m_1}q_{2,i}^{m_2}\ldots q_{n-2,i}^{m_{n-2}}q_{n-1,i}^jq_{n,i}^{j-m}=1\]
for $1\leq i\leq n$. The cases where $1\leq i\leq n-2$ correspond to the condition that $c$ is $\alpha^m\gamma^{-j}$-normal and the cases where $i=n-1, n$ correspond to the conditions that $\gamma(c)=q_{n-1,n}^mc$ and $\alpha(c)=q_{n-1,n}^jc$. It follows that the statements in Theorem~\ref{s-thm}(i) and (ii) are consequences of
of ($*$). Moreover if (i) and (ii) hold and $m_1,m_2,\ldots,m_n\in \Z$ are such that
$q_{1,i}^{m_1}q_{2,i}^{m_2}\ldots q_{n,i}^{m_n}=1$ for $1\leq i\leq n$ then $m_n=m_{n-1}=0$ by (ii) and
$m_1=m_2=\ldots=m_{n-2}=0$ by (i).

  Note that if $\mathcal{O}_Q(T_n)$ is simple then, although $\mathcal{O}_Q(T_{n-2})$ must be $\{\alpha,\gamma\}$-simple, it need not be simple. For an example, take $n=4$ and
\[Q=\begin{pmatrix} 1&1&2&3\\
                    1&1&5&7\\
                    2^{-1}&5^{-1}&1&11\\
                    3^{-1}&7^{-1}&11^{-1}&1\\
\end{pmatrix},\]
where $\mathcal{O}_Q(T_{2})$ is a commutative Laurent polynomial ring in two indeterminates.
\end{example}

\begin{example}[{\bf Higher Weyl algebras and quantized Weyl algebras}]
Let $n\geq 1$, let $\q=(q_{1},\ldots,q_{n})$ be an $n$-tuple of elements
of $\F$, and let $\Lambda=(\lambda_{i,j})$ be a $n\times n$ matrix over $\F$, with non-zero entries, such that
$\lambda_{j,i}=\lambda_{i,j}^{-1}$ for $1\leq i,j\leq n$ and $\lambda_{i,i}=1$ for $1\leq i\leq n$.
For $1\leq i<j\leq n$, let
$\mu_{i,j}=q_{i}\lambda_{i,j}$ and $\mu_{j,i}=\mu_{i,j}^{-1}=q_{i}^{-1}\lambda_{j,i}$.
 The $n${\it th quantized
Weyl algebra} $A_n^{\Lambda,\q}$ given by these data is the
$\F$-algebra generated by
$y_1,x_1,\ldots,y_n,x_n$ subject to the relations
\begin{eqnarray*}
y_{j}y_{i}&=&\lambda_{j,i}y_{i}y_{j},\quad 1\leq i<j\leq n;\\
y_{j}x_{i}&=&\lambda_{i,j}x_{i}y_{j},\quad 1\leq i<j\leq n;\\
x_{j}y_{i}&=&\mu_{i,j}y_{i}x_{j},\quad 1\leq i<j\leq n;\\
x_{j}x_{i}&=&\mu_{j,i}x_{i}x_{j},\quad 1\leq i<j\leq n;\\
x_{j}y_{j}-q_{j}y_{j}x_{j}&=&v_{j-1},\quad 1\leq j\leq n.
\end{eqnarray*}
where
\[
v_0=1,\text{ and, for }2\leq j\leq n+1\text{, }v_{j-1}=1+\sum_{i=1}^{j-1}(q_{i}-1)y_{i}x_{i}.\]

In the case where $q_i\neq 1$ for $1\leq i\leq n$, it is shown in \cite[2.8]{qweyl} that $A_n^{\Lambda,\q}$ is obtained from $\F$ by $n$ iterations of the ambiskew polynomial ring construction, with
$A_n^{\Lambda,\q}=A(A_{n-1}^{\Lambda,\q},\alpha_n,v_{n-1},q_n)$, where
for $1\leq i\leq n-1$,
\[\alpha_n(y_{i})=\lambda_{n,i}y_{i}\text{ and }
                        \alpha_n(x_{i})=\lambda_{i,n}x_{i}.\]
In this case, each step is conformal and $v_{i-1}$ is a scalar multiple of the Casimir element of $A_{i-1}^{\Lambda,\q}$.
The normal elements $v_0,v_1,\ldots,v_{n}$ of $A_n^{\Lambda,\q}$ satisfy the equations
\[
v_jy_{i}=\begin{cases}
q_{i}y_{i}v_j\mbox{ if }i<j\\
y_{i}v_j\mbox{ if }i\geq j,\end{cases}\text{ and }
v_jx_{i}=\begin{cases}
q_{i}^{-1}x_{i}v_j \mbox{ if }i<j\\
x_{i}v_j\mbox{ if }i\geq j\end{cases}
\]
for $1\leq i,j\leq n$. Hence $v_iv_j=v_jv_i$ for $0\leq i,j\leq n$.

If $q_i=1$ for some $i$ then the construction remains valid
but the step from $A_{i-1}^{\Lambda,\q}$ to $A_i^{\Lambda,\q}$ is singular. To see this, note that $A_n^{\Lambda,\q}$
has a $\Z$-grading such that $\deg y_i=1$ and $\deg x_i=-1$ for $1\leq i\leq n$, that each $v_j$ is homogeneous of degree $0$, that $\alpha_n$ is $\Z$-graded and that $\alpha(a)=a$ for homogeneous elements of degree $0$. It follows that the degree $0$ component of $u-\alpha(u)$ is $0$ for all $u\in A_{i-1}^{\Lambda,\q}$ whereas the degree $0$ component of $v_i$ is $v_i$ itself. Note also that, in this case $v_{i}=v_{i-1}$ and,
for each $j\neq i$,
$\mu_{i,j}=\lambda_{i,j}$. In particular, if $q_i=1$ for all $i$ then each $v_i=1$ and the relations are
\begin{eqnarray*}
y_{j}y_{i}&=&\lambda_{j,i}y_{i}y_{j},\quad 1\leq i<j\leq n;\\
y_{j}x_{i}&=&\lambda_{i,j}x_{i}y_{j},\quad 1\leq i<j\leq n;\\
x_{j}y_{i}&=&\lambda_{i,j}y_{i}x_{j},\quad 1\leq i<j\leq n;\\
x_{j}x_{i}&=&\lambda_{j,i}x_{i}x_{j},\quad 1\leq i<j\leq n;\\
x_{j}y_{j}-y_{j}x_{j}&=&1,\quad 1\leq j\leq n.
\end{eqnarray*}
In this case we write $A_n^{\Lambda,\q}$ as $A_n^{\Lambda}$. If each $\lambda_{i,j}=1$ then $A_n^{\Lambda}$
is the usual $n${\it th Weyl algebra} $A_n(\F)$. In \cite{qweyl}, where it was assumed that $q_i\neq 1$ for all $i$, it was shown that, in all characteristics, the localization of $A_{n}^{\Lambda,\q}$ at the multiplicatively closed set $\mathcal{V}$ generated by the $n$ normal elements
$v_1,v_2,\ldots,v_{n}$ is simple provided no $q_i$ is a root of unity. The companion result for the case where $\ch \F=0$ and each $q_i=1$ is the following generalization of the simplicity of the usual Weyl algebras.

\begin{theorem}\label{Asimple} If $\ch \F=0$ then  $A_n^{\Lambda}$ is simple for all $\Lambda$.
\end{theorem}
\begin{proof}
This follows inductively from Theorem~\ref{r-thm0}. At each step, $v=1$ and $v^{(m)}=m$  and, as observed above, the ambiskew extension is singular.
Alternatively the induction may be carried out using Theorem~\ref{simpleit}, with $\rho=\mu=1$.
\end{proof}

\begin{rmk}
The algebras $A_n^{\Lambda}$ can be interpreted in terms of quantum differential operators in the sense of Lunts and Rosenberg \cite{LR}, see \cite{IJM}.
\end{rmk}

Next we consider hybrid cases where some, but not all, of the parameters $q_i$ are $1$.
\begin{notn}
With $\q$, $\Lambda$ and $\mathcal{V}$ as above, let $B_{n}^{\Lambda,\q}$ denote the localization of
$A_{n}^{\Lambda,\q}$ at $\mathcal V$.
In particular, if $q_i=1$ for all $i$ then $B_{n}^{\Lambda,\q}=A_{n}^{\Lambda}$. As $\alpha_n(v_i)=v_i$ for
$0\leq i\leq n-1$, the automorphism $\alpha_n$
of $A_{n-1}^{\Lambda,\q}$ extends to an automorphism of $B_{n-1}^{\Lambda,\q}$ and, by \cite[Lemma 1.4]{g}, localization at the powers of each $v_i$ commutes with each of the subsequent skew polynomial extensions. In particular, if $q_n\neq 1$ then
$B_{n}^{\Lambda,\q}$ is the localization of  $R(B_{n-1}^{\Lambda,\q},\alpha_n,v_{n-1},q_n)$ at the powers of $v_{n}$.
If $q_n=1$ then
$B_{n}^{\Lambda,\q}=R(B_{n-1}^{\Lambda,\q},\alpha_n,v_{n-1},1)$.
\end{notn}

The following Lemma will be used to keep track of the units through the various localizations in the construction of $B_{n}^{\Lambda,\q}$.
\begin{lemma}\label{units} Let $B$ be a right Noetherian ring with a regular normal element $v$ such that $vB$ is a completely prime ideal and let $\mathcal{U}=\{uv^j:u\in U(B),j\in\Z\}$. Let $C$ be the localization of $B$ at the right and left Ore set $\{v^i:i\geq 0\}$. Then $U(C)=\mathcal{U}$.
\end{lemma}
\begin{proof} The result is trivial if $v$ is a unit in $B$ so we can assume that $v\notin U(B)$. Let
$\mathcal{U}^\prime=\mathcal{U}\cap B=\{uv^j:u\in U(B),j\geq 0\}$ and note that, by the normality of $v$, $v\mathcal{U}^\prime\subseteq \mathcal{U}^\prime$.
Suppose that the result is false. Then  there exist $a\in B$ and $i\geq 0$ such that $av^{-i}\in U(C)\backslash \mathcal{U}$.  In particular,
$a\notin \mathcal{U}^\prime$, otherwise $av^{-i}$ has the form $uv^{j-i}\in \mathcal{U}$. Also $a\in U(C)$. Let $bv^{-\ell}$ be the inverse of $a$ in $C$ and note that $\ell>0$, otherwise $a\in U(B)\subset \mathcal{U}^\prime$.
Thus $ab=v^\ell$.   Also $b\notin \mathcal{U}^\prime$, for if $b=uv^j$ for some $u\in U(B)$ and $j\geq 0$
then
\[av^{-i}=v^{\ell}b^{-1}v^{-i}=v^{\ell-j}u^{-1}v^{-i}=wv^{\ell-j-i}\in \mathcal{U}\]
  for some $w\in U(B)$.

  Now let $\ell>0$ be minimal such that there exist $a,b\in B$ such that $ab=v^\ell$ but
  $a\notin \mathcal{U}^\prime$ and $b\notin \mathcal{U}^\prime$. As $vB$ is completely prime, either $a\in vB$ or $b\in vB$. Suppose that $a=va^\prime\in vB$. Then $va^\prime b=v^{\ell}$ so $a^\prime b=v^{\ell-1}$. By the minimality of $\ell$, either $\ell=1$, in which case, by \cite[Corollary 4.25]{GW}, $b\in U(B)\subset \mathcal{U}^\prime$, or $a^\prime \in \mathcal{U}^\prime$, in which case $a=va^\prime\in v\mathcal{U}^\prime\subseteq \mathcal{U}^\prime$. Each case gives rise to a contradiction. A similar argument on the other side yields a contradiction if $b\in vB$. Therefore the result is true.
\end{proof}

\begin{theorem}\label{Bsimple}
If $\chr \F=0$ then $B_{n}^{\Lambda,\q}$
is simple if and only if, for all $i$ such that $1\leq i\leq n$, either $q_i=1$ or $q_i$ is not a root of unity.
\end{theorem}
\begin{proof}
Suppose that, for $1\leq i\leq n$, either $q_i=1$ or $q_i$ is not a root of unity. The proof that $B_{n}^{\Lambda,\q}$
is simple is by induction. Note that $B_{1}^{\Lambda,\q}$, which is either the usual Weyl algebra or the localization of $A_1^{q_1}$ at the powers of the Casimir element, is simple  from the discussion in
Subsection~\ref{fromk}. Let $n>1$ and suppose that $B_{n-1}^{\Lambda,\q}$ is simple. Let $q=q_n$, $\alpha=\alpha_n$
and $v=v_{n-1}$.

First suppose that $q=1$. Then $B_{n}^{\Lambda,\q}=R(B_{n-1}^{\Lambda,\q},\alpha,v,1)$.  Note that $B_{n-1}^{\Lambda,\q}$ is the localization of $A_{n-1}^{\Lambda,\q}$ at the multiplicatively closed set $\mathcal{V}$ generated by $v_1,v_2,\ldots, v_{n-1}$ and that $\alpha(v_i)=v_i$ for $1\leq i\leq n-1$.
Let $u\in A_{n-1}^{\Lambda,\q}$ and $w\in \mathcal{V}$ be such that $uw^{-1}-\alpha(uw^{-1})=v$.
Then $\alpha(w)=w$ and $u-\alpha(u)=vw$. In the $\Z$-grading of $A_{n-1}^{\Lambda,\q}$, $v$ and $w$ are homogeneous of degree $0$, whereas the component in  degree $0$ of $u-\alpha(u)$ is $0$. Hence $(B_{n-1}^{\Lambda,\q},\alpha,v,1)$ is singular.
Also, for $m\geq 1$,
$v^{(m)}=mv$ which, as  $v$ is already a unit in $B_{n-1}^{\Lambda,\q}$, is a unit.
By Theorem~\ref{r-thm0},  $B_{n}^{\Lambda,\q}$ is simple in this case.

 Now suppose that $q\neq 1$, in which case Theorem~\ref{s-thm2} is applicable with $R=B_{n-1}^{\Lambda,\q}$, which is simple, and $S=B_{n}^{\Lambda,\q}$. Here $v^{(m)}=\qnum{m}{q}v$, which is a unit as $q$ is not a root of unity, so statements (i) and (iii) of Theorem~\ref{s-thm2} both hold. For $1\leq j\leq n-1$, $A_{j}^{\Lambda,\q}/v_jA_{j}^{\Lambda,\q}$ is a domain, being a generalized Weyl algebra over $A_{j-1}^{\Lambda,\q}$. It follows from Lemma~\ref{units} that the group $U(B_{n-1}^{\Lambda,\q})$ is generated by $\F^*$ and $\{v_i^{\pm 1}:1\leq i\leq n-1\}$. Hence $\alpha(c)=c$ for all $c\in
 U(B_{n-1}^{\Lambda,\q})$. For statement (ii), a $(0,-j)$-special element $c$ would be a unit in $B_{n-1}^{\Lambda,\q}$
satisfying $\alpha(c)=q^{j}(c)$, where $j>0$. As $q$ is not a root of unity,  no $(0,-j)$-special element $c$ exists and hence, by Theorem~\ref{s-thm2}, $B_{n}^{\Lambda,\q}$ is simple in this case also.

Conversely, suppose that there exists $j$, $1\leq j\leq n$, such that $q_j\neq 1$ but $q_j^m=1$ for some $m>0$.
Consider $y_j^m$, which is a non-unit both in $A_{n}^{\Lambda,\q}$ and $B_{n}^{\Lambda,\q}$.
If $i\neq j$ then $y_iy_j^m=\lambda_{i,j}^m y_j^my_i$ and $x_iy_j^m=\lambda_{j,i}^m y_j^mx_i$.
Also, by \eqref{skewcomm}, \[x_jy_j^m=q_j^my_j^mx_j+v_j^{(m)}y_j^{m-1}=q_j^my_j^mx_j\] as
$v_{j-1}^{(m)}=\qnum{m}{q_j}v_{j-1}=0$. Hence $y_j^m$ is a normal non-unit in $B_{n}^{\Lambda,\q}$ which can therefore
not be simple.
\end{proof}

\begin{rmk}
A similar result is true for the alternative quantized Weyl algebra $\mathscr{A}_n^{\Lambda,\q}$ studied in \cite{akhav} where the relations are
\begin{eqnarray*}
y_{j}y_{i}&=&\lambda_{j,i}y_{i}y_{j},\quad 1\leq i<j\leq n;\\
y_{j}x_{i}&=&\lambda_{i,j}x_{i}y_{j},\quad 1\leq i<j\leq n;\\
x_{j}y_{i}&=&\lambda_{i,j}y_{i}x_{j},\quad 1\leq i<j\leq n;\\
x_{j}x_{i}&=&\lambda_{j,i}x_{i}x_{j},\quad 1\leq i<j\leq n;\\
x_{j}y_{j}-q_{j}y_{j}x_{j}&=&1,\quad 1\leq j\leq n.
\end{eqnarray*}
When $q_i=1$ for all $i$, $\mathscr{A}_n^{\Lambda,\q}={A}_n^{\Lambda,\q}$.
\end{rmk}
\end{example}
\begin{example}[{\bf A simple ambiskew polynomial ring in characteristic $p$}]
For an example of a simple ambiskew polynomial ring in characteristic $p$, take $A$ to the field of rational functions,
over a field $\F$ of characteristic $p$, in countably many indeterminates $t_i$, $i\in \Z$, and let $\alpha$ be
the $\F$-automorphism of $A$ such that $\alpha(t_i)=t_{i+1}$ for all $i\in \Z$. Let $\rho\in \F^*$ and let
$v\in \F(t_0)\backslash \F$. It is clear that $A$, being a field, is $\alpha$-simple and that, for all $m\geq 1$, $v^{(m)}$ is a unit of $A$. Suppose that condition (ii) of Theorem~\ref{r-thmp} fails due to the existence of $n$, $u$ and $b_i$, $0\leq i\leq n-1$, with the precluded properties. As $\alpha(b_i)=\rho^{p^i-p^n}b_i$, it is clear that each $b_i\in \F$.  Then, for all $n\geq 0$, $v^{p^n}+\sum_{0}^{n-1}b_iv^{p^i}\in \F(t_0)\backslash \F$
whereas, for all $u\in A$ either $u\in \F$, in which case $\rho^{p^n}\alpha(u)-u\in \F$, or
$\rho^{p^n}\alpha(u)-u\notin \F(t_0)$. In each case this contradicts the failure of (ii).
Therefore $R(A,\alpha,v,\rho)$ is simple by Theorem~\ref{r-thmp}.

\end{example}

\begin{example}[{\bf Ambiskew polynomial rings over $\F[t]$ and $\F[t^{\pm 1}]$}]
\label{ALaurent}
Given any $\F$-automorphism $\alpha$ of the polynomial algebra $\F[t]$ the generator
$t$ can be chosen so that $\alpha(t)=t+1$ or $\alpha(t)=\lambda t$ for some $\lambda\in \F^*$.

(i) Suppose that $\alpha(t)=t+1$.
If $\chr \F=p>0$ then  $R(\F[t],\alpha,v,\rho)$ need not be conformal. For example, if $\chr \F=2$ then
$R(\F[t],\alpha,t,1)$ is singular. However $t^p-t$ generates an $\alpha$-ideal so the rings $R=R(\F[t],\alpha,v,\rho)$ and, when they exist,
$R_\mathcal{Z}$ and $R/zR$ are never simple when $\chr \F=p>0$. So we restrict to the case where $\chr \F=0$, in which
$\F[t]$ is $\alpha$-simple and $R(\F[t],\alpha,v,\rho)$ is always conformal with $\deg (u)=\deg (v)+1$ if $\rho=1$
and $\deg (u)=\deg (v)$ if $\rho\neq 1$.
When $\rho=1$ these are the Smith algebras introduced in \cite{psmith}.

By Remark~\ref{rhoone1} and Theorem~\ref{s-thm2}, the localization $S$ is not simple when $\rho$ is a root of unity. If $v\in \F^*$
and $\rho$ is not a root of unity then $S$ is simple by Theorem~\ref{s-thm2}, as there is no  $(0,-j)$-special element with $j\geq 1$
 and $v^{(m)}=\qnum{m}{\rho}v$ is always a unit.

 Finally suppose that $\deg v \geq 1$ and that $\rho$ is not a root of unity. Then $d:=\deg v \geq 1$.
 Let $v_d$ be the coefficient of $t^d$ in $v$. Then the coefficient of $t^d$ in $u$ is $v_d/(1-\rho)$ and, for
 all $m\geq 1$,
 the coefficient of $t^d$ in $v^{(m)}$ is $v_d\qnum{m}{\rho}$. Thus
 $\deg v^{(m)}=d>0$ and each $v^{(m)}$ has an irreducible factor.
 Let $f$ be an irreducible factor in $\F[t]$ of $u$. Then there exists $m\geq 1$ such that $f$ is not a factor of $\alpha^m(u)$. Otherwise
 $u\in I:=\cap_{m\geq 0} \alpha^{-m}(fA)$ and $\alpha(I)\subseteq I$, contradicting the $\alpha$-simplicity of $\F[t]$.
 As $u$ has only finitely many irreducible factors, up to associates,  there exists $M\geq 1$ such that $u$ and $\alpha^M(u)$ are coprime in $\F[t]$. But, by Theorem~\ref{s-thm2}, $u^n\in v^{(M)}\F[t]$ for some $n\geq 0$ so that any irreducible factor of $v^{(M)}=u-\rho^M\alpha^M$ must be a common factor of $u$ and $\alpha^M(u)$. This is a contradiction so $S$ is not simple in this case. Thus $S$ is simple if and only if $v\in \F^*$
and $\rho$ is not a root of unity.

(ii) When
 $\alpha(t)=qt$ for some $q\in \F^*$, $\F[t]$ is clearly not $\alpha$-simple so we will
consider here the ambiskew polynomial rings over the Laurent polynomial ring $\F[t^{\pm 1}]$ determined by $\alpha$. Moreover, we shall assume that $q$
is not a root of unity so that $\F[t^{\pm 1}]$ is $\alpha$-simple. Let $v=\sum_{i=r}^s a_it^i\in \F[t^{\pm 1}]$ and let $\rho\in \F^*$. As $\alpha$ is $\Z$-graded, any splitting element $u$ must have the form $\sum_{i=r}^s b_it^i$, in which case
$u-\rho\alpha(u)=\sum_{i=r}^s(1-q^i\rho) b_it^i$. Therefore $(\F[t^{\pm 1}],\alpha,v,\rho)$ is singular if and only if
there exists $n\in \Z$ such that $a_n\neq 0$ and $\rho=q^{-n}$. For $R(\F[t^{\pm 1}],\alpha,v,\rho)$ to be simple, $v$ must be a unit so set $v=at^n$, where $a\in \F^*$, $0\neq n\in \Z$ and $\rho=q^{-n}$. Without loss of generality, we can, by replacing $x$ by $a^{-1}x$, assume that $a=1$. The relations in $R$
are
\[yt=qty,\; xt=q^{-1}tx\text{ and }xy-q^{-n}yx=t^n.\] Note that $\rho\alpha(v)=v$, so $v^{(m)}=m$ for $m\geq 1$.
By Theorem~\ref{r-thm0}, $R:=R(\F[t^{\pm 1}],\alpha,t^n,q^{-n})$ is simple if $\chr \F=0$.
On the other hand, if
 $\chr \F=p\neq 0$ then $v^{(p)}=0$ and, by Theorem~\ref{r-thm0}, $R$ is not simple.

  \end{example}

\begin{example}[{\bf quantized Heisenberg algebras}]
\label{Heisenberg}
Here we discuss a two-parameter quantization of the enveloping algebra of
the three-dimensional Heisenberg algebra over $\F$, namely
the $\F$-algebra
$U_{q,\rho}$ generated by $t,x$ and $y$ subject to the relations
\[yt=qty,\quad xt=q^{-1}tx, \quad xy-\rho yx=t\]
and its localization $U^\prime_{q,\rho}$ at the powers of $t$.
Thus $U_{1,1}$ is the enveloping algebra of
the three dimensional Heisenberg algebra over $\F$. Here
$U_{q,\rho}=R(\F[t],\alpha,t,\rho)$ and $U^\prime_{q,\rho}=R(\F[t^{\pm 1}],\alpha,t,\rho).$
In \cite{ks}, two one-parameter
quantizations or `$q$-analogues' of $U_{1,1}$  are discussed and contrasted. In our notation, these are $U_{q,q}$ and $U_{q,q^{-1}}$.
From Example~\ref{ALaurent}(ii) with $n=1$, we can see that $U^\prime_{q,\rho}$ is simple
if and only if $\rho=q^{-1}$. Hence  every non-zero ideal of $U_{q,q^{-1}}$ contains a power of $t$ whereas, in
$U_{q,q}$, the Casimir element $z=xy-(1-q^2)^{-1}t$ generates a non-zero proper ideal.

Suppose that $\rho q\neq 1$ so that  $U_{q,\rho}$ and $U^\prime_{q,\rho}$ have a Casimir element
$z=xy-(1-\rho q)^{-1}t$. Let $S$ be the localization of $U^\prime_{q,\rho}$ at the powers of $z$.
We apply Theorem~\ref{s-thm2} to determine when $S$ is simple. Firstly, $\F[t^{\pm 1}]$ is $\alpha$-simple if and only if $q$ is not a root of unity.
Secondly, for $j\in \Z$, there is a  $(0,-j)$-special element if and only if $\rho^jq^i=1$ for some $i\in \Z$, in which case
$\alpha(t^i)=\rho^{-j}t^i$. For all $m\geq 1$, $v^{(m)}=\qnum{m}{\rho q}t$ which is $0$ if
$(\rho q)^m=1$ and is a unit otherwise. By Theorem~\ref{s-thm2}, $S$ is simple if and only if
the subgroup of $\F^*$ generated by $q$ and $\rho$ is free abelian of rank $2$.
In particular, when $\rho=q$, $z-t^{-1}$ is normal in $U^\prime_{q,q}$
and generates a proper ideal of $S$.
\end{example}

For the remainder of the paper, we assume that the field $\F$ has characteristic $0$.
\begin{example}
[{\bf Ambiskew polynomial rings over finite cyclic group algebras}]\label{cyclic}
In all previous examples, the base ring $A$ has been a domain. Here we consider ambiskew polynomial rings over the group algebra $A:=\F C_n$ of the cyclic group $C_n=\langle s\rangle$ of order $n$.
 Fix a positive integer $n$ for which $\F$ contains a primitive root of unity and fix one such root of unity $\varepsilon$. Let $\alpha$ be the $\F$-automorphism of $\F C_n$
such that $\alpha(s)=\varepsilon s$. Let $v=c_0+c_1s+c_2s^2+\ldots+c_{n-1}s^{n-1}\in A$ and $\rho\in \F^*$.
The ambiskew polynomial ring $R(A,\alpha,v,\rho)$ is  the $\F$-algebra $R_1$ generated by $s, x$ and $y$ subject to the relations
\begin{eqnarray*}
s^n=1,\quad ys&=&\varepsilon sy,\quad xs=\varepsilon^{-1} sx,\\
xy-\rho yx&=&c_0+c_1s+c_2s^2+\ldots+c_{n-1}s^{n-1}.
\end{eqnarray*}
When $\rho=v=1$, $R$ is the skew group algebra $A_1(\F)*C_n$ for the action of $C_n$ in which the generator $s$ acts by $y\mapsto \varepsilon y$ and $x\mapsto \varepsilon^{-1}x$.
When $\rho=1$, $R$ is, in the notation of \cite{cbh}, the algebra $\mathcal{S}^\lambda$ in the case where $\Gamma=C_n$ and $\lambda=v$. It is also the symplectic reflection algebra for the diagonal action of $C_n$ on
$\mathfrak{h}\oplus\mathfrak{h}^*$, where $\mathfrak{h}$ is a $1$-dimensional vector space over $\F$ \cite{EG}.
From  \cite{cbh} we know that, when $\rho=1$, $R$ is simple for generic values of the parameters. We shall see that this is true if $\rho^n=1$ and that $R$ is never simple if $\rho^n\neq 1$. In the statement and proof of the following result the subscripts $i+d$ should be interpreted  modulo $n$.
\end{example}
\begin{prop}\label{simpleRCn}
With the above notation, let $R=R(A,\alpha,v,\rho)$. Then $R$ is simple if and only if there exists $i$, $0\leq i\leq n-1$, such that $\rho=\varepsilon^{-i}$,
$c_i\neq 0$ and \begin{equation}
\label{mc}
mc_i\neq-\sum_{d=1}^{n-1}\qnum{m}{\varepsilon^d}c_{i+d}\varepsilon^{ld}\end{equation}
 for all $m\geq 1$ and all $l$, $0\leq l\leq n-1$.
\label{simpleRCnprop}
\end{prop}
\begin{proof}
First note $A$ is $\alpha$-simple. To see this, let $I$ be a non-zero $\alpha$-ideal, let  $0\neq a=a_0+a_1s+a_2s^2+\ldots+a_{n-1}s^{n-1}\in I$ and let $e(a)=|\supp(a)|$ where
$\supp(a):=\{i:a_i\neq 0\}$. Multiplying by $s^i$ for some $i$ if necessary, we can assume that $0\in \supp(a)$.
Then $a-\alpha(a)\in I$, $e(a-\alpha(a))<e(a)$ and $e(a-\alpha(a))>0$ if $e(a)>1$. If $a\in I\backslash\{0\}$ is chosen with $s(a)$ minimal, it follows that $e(a)=1$ and hence that $I=A$.

Let $u=b_0+b_1s+b_2s^2+\ldots+b_{n-1}s^{n-1}\in A$. Then $u-\rho\alpha(u)=\sum_{i=0}^{n-1} (1-\rho\varepsilon^i)b_is^i$. It follows that $(A,\alpha,v,\rho)$ is conformal unless, for some $i$, $0\leq i\leq n-1$,
$\rho=\varepsilon^{-i}$ and $c_i\neq 0$. In the singular case, for $m\geq 1$ and $0\leq j\leq n-1$, the coefficient of $s^j$ in $\rho^m\alpha^m(v)$ and $v^{(m)}$ are, respectively, $(\varepsilon^{(j-i)m})c_j$ and $\qnum{m}{\varepsilon^{(j-i)}}c_j$. In particular, the coefficient of $s^i$ in $v^{(m)}$ is $mc_i$.
Thus
\begin{equation*}
v^{(m)}=f_m(s):=
\left(mc_i+\sum_{d=1}^{n-1}\qnum{m}{\varepsilon^d}c_{i+d}s^d\right)s^i.
\end{equation*}
As $v^{(m)}$ is a unit if and only if $f_m(\varepsilon^l)\neq 0$ for  $0\leq l\leq n-1$, the result now follows from Theorem~\ref{r-thm0}.
\end{proof}

\begin{rmk}
When $t\neq 0$ the sequence
$(\qnum{m}{\varepsilon^{t}})_{m\geq 0}$ is periodic, with period dividing $n$, and takes the value $0$ whenever $n$ divides $m$.
Therefore as $m$ varies, the set $F$ of values that can be taken by the right hand side of \eqref{mc}
is finite and includes $0$.
In particular $v^{(m)}=mc_is^i$ whenever $n$ divides $m$.
\end{rmk}

\begin{example}
\label{sra1}
In Example \ref{simpleRCn}, let $\rho=1$.  By Proposition~\ref{simpleRCnprop}, $R$ is simple if and only if
$c_0\neq 0$ and \[mc_0\neq-\sum_{j=1}^{n-1}\qnum{m}{\varepsilon^j}c_{j}\varepsilon^{lj}\]
 whenever $m\geq 1$ and $0\leq l\leq n-1$.

\end{example}
\begin{example}\label{R1n2}
In Example \ref{simpleRCn}, let $n=2$, so that $\varepsilon=-1$, the values of $\rho$ for which $R_1$ can be simple are $1$ and $-1$, and $v=c_0+c_1s$. Note that $\qnum{m}{-1}=1$ if $m$ is odd and $\qnum{m}{-1}=0$ if $m$ is even.
When $\rho=1$, \[f_m(s)=\begin{cases}
mc_0+c_1s\text{ if }m\text{ is odd}\\
mc_0\text{ if }m\text{ is even}\end{cases}\]
so $R$ is simple if and only if $c_0\neq 0$ and $c_1\neq \pm mc_0$ for all odd $m\in \N.$
 When $\rho=-1$, the relations for $R_1$ become
\[
s^2=1,\quad ys=\varepsilon sy,\quad xs=\varepsilon^{-1} sx, \quad xy+yx=c_0+c_1s
\]
and
\[f_m(s)=\begin{cases}
c_0+mc_1s\text{ if }m\text{ is odd}\\
mc_1s\text{ if }m\text{ is even}\end{cases}\]
so $R$ is simple if and only if $c_1\neq 0$ and $c_0\neq \pm mc_1$ for all odd $m\in \N.$
\end{example}

\begin{example}\label{R2n}
Let $R(\F C_n,\alpha,v,\rho)$ be as in Example~\ref{simpleRCn} and, with a view to iterating the ambiskew polynomial construction,
rename $x,y,v,\rho,\alpha$ and $R$ as $x_1,y_1,v_1,\rho_1,\alpha_1$ and $R_1$. Suppose that $R_1$ is simple, so that
$\rho_1=\varepsilon^{-i}$ for some $i$.  Suppose also that, for some $g, j$ with $0\leq g,j \leq n-1$, $\alpha_1^j(v_1)=\varepsilon^g v_1$.  The values $j=0$ and $g=0$ work for any choice of $v_1$ but other possibilities will be explored in Examples~\ref{R21} and \ref{R22} below.   Let $\lambda,\rho_2\in \F^*$. There is an $\F$-automorphism $\alpha_2$ of $R_1$
such that $\alpha_2(s)=\alpha_1^j(s)=\varepsilon^js$, $\alpha_2(y_1)=\lambda y_1$  and $\alpha_2(x_1)=\varepsilon^g\lambda^{-1} x_1$.
The choice of $v_2$, compared to the choice of $v_1$ in the construction of $R_1$ from the commutative algebra $\F C_n$, is restricted by the requirement that $v_2$ should be normal.
Let $0\leq h\leq n-1$, let $d\in \F^*$ and let
$v_2=ds^h$. Then $v_2$ is normal in $R_1$ inducing the inner automorphism $\gamma$ of $R_1$ such that
$\gamma(y_1)=\varepsilon^{-h}y_1$,
$\gamma(x_1)=\varepsilon^hx_1$ and $\gamma(s)=s$ while $\beta:=\gamma\alpha_2^{-1}$ is such that
$\beta(y_1)=\varepsilon^{-h}\lambda^{-1}y_1$,
$\beta(x_1)=\varepsilon^h\lambda x_1$ and $\beta(s)=s$.
The defining relations for $R:=R(R_1,\alpha_2,
v_2,\rho_2)$ are
\begin{align*}
s^n&=1,\\
y_1s=\varepsilon sy_1,&\quad x_1s=\varepsilon^{-1} sx_1,\\
y_2s=\varepsilon^j sy_2,&\quad x_2s=\varepsilon^{-j} sx_2,\\
x_2x_1=\varepsilon^h\lambda x_1x_2,&\quad x_2y_1=\varepsilon^{-h}\lambda^{-1} y_1x_2,\\
y_2x_1=\lambda^{-1} x_1y_2,&\quad y_2y_1=\lambda y_1x_2,\\
x_1y_1-\rho_1 y_1x_1&=c_0+c_1s+c_2s^2+\ldots+c_{n-1}s^{n-1},\\
x_2y_2-\rho_2y_2x_2&=ds^h.
\end{align*}
If $(R_1,\alpha_2,v_2,\rho_2)$ is conformal then the splitting element $u$, such that $v_2=u-\rho_2\alpha_2(u)$, must have the form $u=cs^h$ where $d=c(1-\varepsilon^{gh}\rho_2)$. Hence $(R_1,\alpha_2,v_2,\rho_2)$ is singular if and only if
$\rho_2=\varepsilon^{-gh}$. When $\rho_2=\varepsilon^{-gh}$, $v_2^{(m)}=mds^h$ is a unit for all $m\geq 1$ so, by Theorem~\ref{r-thm0},
$R_2=R(R_1,\alpha_2,ds^h,\varepsilon^{-gh})$ is simple.
\end{example}

\begin{example}\label{sra2}
One particular example of the construction of $R$ in Example~\ref{R2n} yields, after a change of generators, another symplectic reflection algebra. The symplectic reflection algebra $H_{t,c}(S_2)$ for the diagonal action of $S_2$ on $\mathfrak{h}\oplus\mathfrak{h}^*$, where $\mathfrak{h}=\F^2$ and $s$ acts on $\mathfrak{h}$ by transposition $(a,b)\mapsto (b,a)$ is generated by $s$, $w_1$, $z_1$, $w_2$, and $z_2$ subject to the following relations \cite{EG}:
\begin{align*}
s^2&=1,\\
sw_1=w_2s,&\quad sz_1=z_2s,\\
w_1w_2=w_2w_1,&\quad z_1z_2=z_2z_1,\\
w_1z_1=z_1w_1-t+cs,&\quad
w_2z_2=z_2w_2-t+cs,\\
w_1z_2=z_2w_1-cs,&\quad
w_2z_1=z_1w_2-cs,
\end{align*}
where $t,c\in \F$. If we change generators to $s, x_1:=(z_1+z_2)/2,
y_1:=(w_1-w_2)/2, x_2:=(z_2-z_1)/2$ and $y_2:=(w_2-w_1)/2$ then the relations become:
\begin{align*}
s^2&=1,\\
x_1s=-sx_1,&\quad
y_1s=-sy_1,\\
x_2s=sx_2,&\quad
y_2s=sy_2,\\
x_2x_1=x_1x_2&\quad
x_2y_1=y_1x_2,\\
y_2x_1=x_1y_2,&\quad
y_2y_1=y_1y_2,\\
x_1y_1-y_1x_1=2t-4cs,&\quad
x_2y_2-y_2x_2=2t.
\end{align*}
To construct this using the method of Example~\ref{R2n}, first form $R_1=R(\F C_2,\alpha_1,2t-4cs,1)$.
Then take $j=g=h=0$, $\lambda=1$, $d=2t$ and $\rho_2=1=\varepsilon^{gh}$ and take $\alpha_2=\id_{R_1}$.
By Example~\ref{R1n2}, $R_1$ is simple if and only if  $t\neq 0$ and $2c\neq \pm mt$ for all odd $m\geq 1$.
As $\rho_2=\varepsilon^{gh}$, it follows from Example~\ref{R2n} that if $R_1$ is simple then so too is $R$. The converse is true by Theorem~\ref{r-thm0}, because $\alpha_2=\gamma=\id$, so $R$ is simple if and only if  $t\neq 0$ and $2c\neq \pm mt$ for all odd $m\geq 1$.

This example can be obtained by iterating the ambiskew construction in an alternative order. If $R_2$ is the subalgebra generated by $x_2$ and $y_2$, then $B_1=R(\F C_2,\id_\F,2t,1)$, which, provided $t\neq 0$, is isomorphic to the group algebra
$A_1(\F)C_2$ which is not simple but is $\tau$-simple for the automorphism $\tau$ of
$R_2$ under which $s\mapsto -s$, $x_2\mapsto x_2$ and $y_2\mapsto y_2$. The element $2t-4cs$ is central in $R_2$ and
$R=R(R_2,\tau,2t-4cs,1)$. It is a routine matter to use Theorem 3.10 to get an alternative proof that
$R$ is simple if and only if  $t\neq 0$ and $2c\neq \pm mt$ for all odd $m\geq 1$.
\end{example}

\begin{example}\label{R21} Here we give specific instances of Example~\ref{R2n} in which $j$ need not be $0$.
Let $R_1=R(\F C_n,\alpha_1,s^a,\varepsilon^{-a})$, where $0\leq a\leq n-1$.
Then \[y_1s=\varepsilon sy_1,\;x_1s=\varepsilon^{-1}sx_1\text{ and }x_1y_1-\varepsilon^{-a}y_1x_1=s^a,\]
$(\F C_n,\sigma_1,s^a,\varepsilon^{-a})$ is singular, each $v^{(m)}=ms^a$ and $R_1$ is simple.
 For any $j$, $0\leq j\leq n-1$, observe that $\alpha_1^j(s^a)=\varepsilon^{aj}s^a$ so, in the notation of Example~\ref{R2n}, $g=aj$.  Let $\lambda\in \F^*$. There is an $\F$-automorphism $\alpha_2$ of $R_1$ such that $\alpha_2(s)=\varepsilon^js$, $\alpha(y_1)=\lambda y_1$ and $\alpha(x_1)=\varepsilon^{aj}\lambda^{-1}x_1$. For each $0\leq b\leq n-1$, $s^b$ is normal in $R_1$ and, for each $\rho\in \F^*$, we can form $R_2=R(R_1,\alpha,s^b,\rho)$.
The extra relations for $R_2$ are:
\begin{align*}
y_2s=\varepsilon^j sy_2,&\quad
x_2s=\varepsilon^{-j} sx_2,\\
y_2x_1=\varepsilon^{aj}\lambda^{-1} x_1y_2,&\quad
y_2y_1=\lambda y_1y_2,\\
x_2x_1=\varepsilon^{a(1-j)}\lambda x_1x_2,&\quad
x_2y_1=\varepsilon^{-a}\lambda^{-1}  y_1x_2,\\
x_2y_2-\rho y_2x_2&=s^b.
\end{align*}
Here $(R_1,\alpha_2,\rho,s^b)$ is singular if $\rho=\varepsilon^{-jb}$ in which case $v^{(m)}=ms^b$ for all $m\geq 1$ and $R_2$ is simple.
\end{example}

\begin{example}\label{R22} Our purpose here is to exhibit an instance of Example~\ref{R2n} in which neither $v_1$ nor $v_2$ is homogeneous in the $C_n$-grading of $\F C_n$.
Let $n=4$, $\varepsilon=\rho=i$ and $v_1=s+\mu s^3$, where $\mu\in \F^*\backslash\{1,-1\}$.
Thus \[y_1s=isy_1,\;x_1s=-isy_1\text{ and }x_1y_1-iy_1x_1=s+\mu s^3.\]
It is easy to check that $(\F C_4,\alpha_1,s+\mu s^3,i)$ is singular and that $v_1$ is a unit with inverse $(1-\mu^2)^{-1}(s^3-\mu s)$.

The sequence $(i^j\alpha_2^j(v))_{j\geq 1}$ has period two and repeating block
 $s+\mu s^3, -s+\mu s^3$. Therefore
 \[v_1^{(m)}=\begin{cases}
 s+ m\mu s^3\text{ if }m\text{ is odd }\\
 m\mu s^3\text{ if }m\text{ is even }
 \end{cases}\]
By Theorem \ref{r-thm0}, $R_1$ is simple if and only if $\mu\neq \pm1/a$ for all odd positive integers $a$.

Note that $\alpha_1^2(v_1)=-v_1=\varepsilon^2 v_1$ so, taking $j=g=2$ in Example~\ref{R2n},
there is, for each $\lambda\in \F^*$, an $\F$-automorphism $\alpha_2$ of $R_1$ such that $\alpha_2(s)=-s$, $\alpha_2(y_1)=\lambda y_1$ and
$\alpha_2(x_1)=-\lambda^{-1} x_1$. Let $v_2=v_1=s+\mu s^3$. If $\gamma$ is the inner automorphism  of $R_1$ induced by $v_2$ and $\beta=\gamma\alpha_2^{-1}$ then
\begin{align*}
\gamma(s)=s,&\quad \beta(s)=-s,\\
\gamma(y_1)=\frac{i}{\mu^2-1}\left(\mu^2+1+2\mu s^2\right)y_1,&\quad
\gamma(x_1)=\frac{-i}{\mu^2-1}\left(\mu^2+1+2\mu s^2\right)x_1,\\
\beta(y_1)=\frac{i\lambda^{-1}}{\mu^2-1}\left(\mu^2+1+2\mu s^2\right)y_1,&\quad
\beta(x_1)=\frac{i\lambda}{\mu^2-1}\left(\mu^2+1+2\mu s^2\right)x_1.
\end{align*}
As $R_1$ is $\Z$-graded with $\deg(y_1)=1, \deg(x_1)=-1$ and $\deg(s)=0$ and $\alpha_2$ respects this grading,
it follows easily that $(R_1,\alpha_2,s+\mu s^3,-1)$ is singular whereas, for example, $(R_1,\alpha_2,s+\mu s^3,i)$ is conformal.
The relations for $R_2:=R(R_1,\alpha_2,s+\mu s^3,-1)$ are the relations above for $R_1$ together with:
\begin{align*}
y_2s=-sy_2,&\quad
x_2s=-sx_2,\\
y_2x_1=-\lambda^{-1}x_1y_2,&\quad
y_2y_1=\lambda y_1y_2,\\
x_2x_1=\frac{i\lambda}{\mu^2-1}\left(\mu^2+1+2\mu s^2\right)x_1x_2,&\quad
x_2y_1=\frac{i\lambda^{-1}}{\mu^2-1}\left(\mu^2+1+2\mu s^2\right)y_1x_2,\\
x_2y_2+y_2x_2&=s+\mu s^3.
\end{align*}
Let $\mu$ be such that $R_1$ is simple, that is $a\mu\neq 1$ for all odd integers $a$. For $m\geq 1$, $v_2^{(m)}=mv$ is a unit so, by Theorem~\ref{r-thm0},
$R_2$ is also simple.
\end{example}

The second author is no longer active in mathematics research. Enquiries and comments should be addressed to the first author.
\end{document}